\documentclass[12pt,a4paper]{amsart}
\usepackage[abbrev]{amsrefs}
\usepackage{amssymb}

\theoremstyle{plain}
  \newtheorem{theorem}{Theorem}[section]
  \newtheorem{lemma}[theorem]{Lemma}
  \newtheorem{corollary}[theorem]{Corollary}
  \newtheorem{proposition}[theorem]{Proposition}
  \newtheorem{claim}[theorem]{Claim}

\theoremstyle{definition}
  \newtheorem{definition}[theorem]{Definition}
  
\theoremstyle{remark}
  \newtheorem{remark}[theorem]{Remark}
  \newtheorem*{ack}{Acknowledgment}

\numberwithin{equation}{section}

\def\supp{\mathop{\mathrm{supp}}\nolimits}
\begin{document}
\title[]{Relaxation of optimal transport problem \\ via strictly convex functions}
\author[]{Asuka Takatsu$^{\ast\dagger\ddagger}$}
\address{\hspace{-10pt}\scriptsize$\ast$ Department of Mathematical Sciences, Tokyo Metropolitan University, Tokyo {192-0397}, Japan.}
\email{asuka@tmu.ac.jp}
\address{\hspace{-10pt}\scriptsize$\dagger$ Mathematical Institute, Tohoku University, Sendai 980-8578, Japan.}
\address{\hspace{-10pt}\scriptsize$\ddagger$ RIKEN Center for Advanced Intelligence Project (AIP), Tokyo 103-0027, Japan.}
\date{\today}
\maketitle
\vspace{-10pt}
\begin{abstract}
An optimal transport problem on finite spaces is a linear program.
Recently, a relaxation of the optimal transport problem via strictly convex functions,  
especially via the Kullback--Leibler divergence,  sheds new light on data sciences.
This paper provides the mathematical foundations and an iterative process based on a gradient descent
for the relaxed optimal transport problem via Bregman divergences.
\end{abstract}

\section{Introduction}
A  \emph{optimal transport problem} discussed in this paper is a variational problem as follows.
Given $C=(c_{ij})_{1\leq i,j \leq N}\in M_{N}(\mathbb{R})$ and $ x=(x_i)_{i=1}^N, y=(y_j)_{j=1}^N \in \mathbb{R}^N$ with 
\[
x_i, y_j \geq 0 \quad \text{for}\ 1\leq i,j\leq N\quad \text{and}\quad \sum_{i=1}^Nx_i=\sum_{i=1}^Ny_i=1,
\]
find $\Pi =(\pi_{ij})_{1\leq i,j \leq N}\in M_N(\mathbb{R})$ minimizing 
\[
\sum_{i,j=1}^N c_{ij} \pi_{ij} 
\]
under the constraints 
\begin{equation}\label{cc}
\pi_{ij}\geq 0, \quad
\sum_{k=1}^N\pi_{ik}=x_i, \quad 
\sum_{k=1}^N \pi_{kj}=y_j, \quad \text{for $1\leq i,j\leq N$}.
\end{equation}
Since this variational problem is a linear program,  a minimizer may lie on the boundary of the constraint set and not be unique.
Furthermore,  a  gradient descent is not useful to find a minimizer.
%

In data sciences,  a relaxation of the optimal transport problem via strictly convex functions achieves substantial success,
where one of pioneering works is a fast algorithm for the relaxed transport problem via the Kullback--Leibler divergence proposed by Cuturi~\cite{Cuturi}. 
The \emph{Kullback--Leibler divergence} between  $\Pi\in M_N(R)$ satisfying~\eqref{cc}  and $x\otimes y$ is defined by
\[
\mathrm{KL}(\Pi, x\otimes y):=\sum_{\substack{1\leq i,j\leq N, \pi_{ij}\neq0}} \pi_{ij} \log \frac{\pi_{ij}}{x_i y_j}.
\]
The fast algorithm is called \emph{Sinkhorn's algorithm} 
since the convergence of this algorithm is attributed  to an iterative process by Sinkhorn~\cites{sinkhorn64}, \cite{sinkhorn} 
(for historical perspective, see~\cite{PC}*{Remark~4.5} and the references therein).
We notice that  if we run  Sinkhorn's  iteration and stop at the finite step,
then the output  $\Pi\in M_N(\mathbb{R})$ may not satisfy~\eqref{cc}.

As another strictly convex functions,   we focus attention on a Bregman divergence, 
which is a generalization of the Kullback--Leibler divergence via a strictly convex, continuous function $f:[0,1]\to \mathbb{R}$ with  $f\in C^1((0,1])$.
The \emph{Bregman divergence} associated to  $f$ between $\Pi\in M_N(R)$ satisfying~\eqref{cc}  and $x\otimes y$ is defined by 
\[
\mathcal{D}_f(\Pi, x\otimes y):=\sum_{\substack{1\leq i,j\leq N, x_iy_j\neq0}}
\Big(f(\pi_{ij})-f(x_iy_j)-f'(x_iy_j)\left(\pi_{ij}-x_iy_j\right)\Big).
\]
If $f(r)=r \log r$,  where by convention $0 \log 0:=0$,
then $\mathcal{D}_f(\Pi, x\otimes y)=\mathrm{KL}(\Pi, x\otimes y)$.
We notice that Sinkhorn's iterative process is not applicable for Bregman divergences other than  the Kullback--Leibler divergence.
In Section~4, we provide an iterative process to find the relaxed minimizer via  Bregman divergences,
where  the output $\Pi\in M_N(\mathbb{R})$ always satisfies~\eqref{cc} even if we stop the iteration at the finite step
(Theorem~\ref{descent}, Corollary~\ref{cordescent}).
The iterative process is  based on a gradient descent.

For a relaxed optimal transport problem, although there are a lot of successful applications, 
mathematical argument is sometimes not rigorous.
After a brief review of the optimal transport problem in Section~2,
we provide the mathematical foundations of the relaxed  optimal transport problem via Bregman divergences  in Section~3.
We first  provide a criterion for a strictly convex function $f$ such that  a relaxed minimizer lies in the interior of the constraint set (Lemma~\ref{int}).
Then  we prove the continuity (Theorem~\ref{dconti}) and the monotonicity (Theorem~\ref{mono}) in the relaxed  optimal transport problem.
Moreover,  we justify a dual relaxed  optimal transport problem 
(Theorem~\ref{inc}).

\section{Optimal transport problem}
We briefly recall some notions in the optimal transport problem.
When it will introduce no confusion, 
we shall use the same notation $\langle \cdot, \cdot \rangle$ for  the standard inner product on Euclidean space and  the Frobenius inner product on the space of matrices of a fixed size.
The norm induced from $\langle \cdot, \cdot \rangle$ is denoted by $\|\cdot\|_2$.

For $N\in \mathbb{N}$ with $N\geq 2$, define 
\begin{align*}
\mathcal{P}_N&:=\left\{ x=(x_i)_{i=1}^N\in \mathbb{R}^N\ \Big|\  x_i \geq 0 \quad \text{for}\ 1\leq i\leq N, \quad\sum_{i=1}^Nx_i=1 \right\},\\
\mathcal{P}_{N \times N}&:=\left\{ \Pi=(\pi_{ij})_{i,j=1}^N \in M_{N}(\mathbb{R}) \ \Big|\  \pi_{ij} \geq 0\quad  \text{for}\ 1\leq i,j \leq N, \quad\sum_{i,j=1}^N \pi_{ij}=1 \right\}.\end{align*}
We call $\Pi\in \mathcal{P}_{N \times N}$ a \emph{coupling} (or \emph{transport plan}) between $x,y\in \mathcal{P}_N$ if 
\begin{equation}\label{constraints}
\sum_{k=1}^N\pi_{ik}=x_i, \quad 
\sum_{k=1}^N \pi_{kj}=y_j, \quad \text{for $1\leq i,j\leq N$}.
\end{equation}
We denote by  $\Pi(x,y)$ the set of couplings between  $x,y\in \mathcal{P}_N$.
Then $ \Pi(x,y)$ is nonempty due to $x \otimes y \in \Pi(x,y)$,  where  $x \otimes y$ is the outer product of $x,y$, that is, 
\[
(x \otimes y)_{ij}=x_i y_j \quad \text{for $1\leq i,j\leq N$}.
\] 
It is easy to see that $\Pi(x,y)$  is a convex compact subset of $(M_{N}(\mathbb{R}), \|\cdot\|_2)$.

For $x\in \mathcal{P}_N$ and $\Pi \in \mathcal{P}_{N \times N}$, define 
\[
\supp x:=\{ i \ |\ 1\leq i\leq N, \ x_i\neq 0\}, \quad
\supp\Pi:=\{ (i,j) \ |\ 1\leq i,j\leq N, \pi_{ij}\neq 0\}.
\]
\begin{lemma}\label{supp}
For  $x,y \in \mathcal{P}_N$ and $\Pi \in \Pi(x,y)$, 
\[
\supp x \times \supp y =\supp x \otimes  y, \qquad\supp\Pi \subset  \supp x \otimes  y.
\]
\end{lemma}
\begin{proof}
It is trivial that $\supp x \times \supp y =\supp x \otimes  y$. 
For $(i, j)\in \supp\Pi $, we find that 
\[
0<\pi_{ij}\leq \sum_{k=1}^N \pi_{ik}=x_{i}, \qquad
0< \pi_{ij}\leq \sum_{k=1}^N\pi_{ik}=y_{j},
\]
which implies $(i, j)\in   \supp x \otimes  y$.
\end{proof}

For $C=(c_{ij})_{1\leq i,j \leq N}\in M_{N}(\mathbb{R})$, 
we define a function $\mathcal{C}:\mathcal{P}_N \times \mathcal{P}_N  \to \mathbb{R}$ by 
\begin{equation}\label{wass}
\mathcal{C}(x,y)
:=\inf_{\Pi \in \Pi(x,y)} \langle C, \Pi\rangle
=\inf_{\Pi \in \Pi(x,y)} \left( \sum_{i,j=1}^N c_{ij} \pi_{ij} \right).
\end{equation}
Since  $\Pi(x,y)$ is compact and  the function  on $(\Pi(x,y), \|\cdot\|_2)$ sending $\Pi$ to  $\langle C, \Pi\rangle$ is continuous, 
there exists a coupling $\Pi\in \Pi(x,y)$ attaining the infimum in~\eqref{wass}.
Such a coupling is called an \emph{optimal coupling} between~$x,y$. 
Throughout this paper,
we choose a matrix $C\in M_{N}(\mathbb{R})$ arbitrarily and fix it unless otherwise indicated.

The following characterization of optimal couplings is well-known.
Although the proof in the case of finite spaces is easy, the direct proof is less common.
For the sake of completeness, we give a direct proof.
Let  $\mathfrak{S}_M$ be the set of permutations on $M$-letters.
\begin{definition}
A subset $S \subset \{1, \cdots, N\}^2$ is called \emph{$C$-cyclically monotone}
if 
\begin{equation}\label{monono}
\sum_{m=1}^M c_{i_m j_m}
\leq
\sum_{m=1}^M c_{i_{\sigma(m)}j_m}.
\end{equation}
holds for
any family $\{(i_m, j_m)\}_{m=1}^M$ of points in $S$ and any $\sigma\in \mathfrak{S}_M$.
\end{definition}
It is easy to see that a subset of a $C$-cyclically monotone set is $C$-cyclically monotone. 
\begin{proposition}{\rm (cf.\,\cite{Vi}*{Theorem~5.10})}\label{cmono}
Given $x,y \in \mathcal{P}_N$, $\Pi \in \Pi(x,y)$ is optimal if and only if $\supp\Pi$ is $C$-cyclically monotone.
\end{proposition}
\begin{proof}
Let  $\Pi\in \Pi(x,y)$ be an optimal coupling.
If there exist $\{(i_m, j_m)\}_{m=1}^M\subset \supp\Pi$
and $\sigma\in \mathfrak{S}_M$  for which \eqref{monono} is not valid,
then, for $\varepsilon:=\min_{1\leq m\leq M }\{ \pi_{i_mj_m} \}$,  we define $\Pi^\varepsilon \in M_{N}(\mathbb{R})$ by
\[
\pi^\varepsilon_{ij}
:=\begin{cases}
\pi_{ij}-\varepsilon   &\text{if\ } (i,j)\in \{ (i_m, j_m)\}_{m=1}^M,\\
\pi_{ij}+\varepsilon   &\text{if\ } (i,j)\in \{ (i_{\sigma(m)},j_m)\}_{m=1}^M, \\
\pi_{ij} & \text{otherwise}.
\end{cases}
\]
We see that $\Pi^\varepsilon\in \Pi(x,y)$ and 
\[
0\leq   \langle C, \Pi^\varepsilon\rangle-  \langle C, \Pi \rangle
=-\varepsilon \sum_{m=1}^M  c_{i_m j_m}+\varepsilon \sum_{m=1}^M  c_{i_{\sigma(m)}j_m}
=\varepsilon \sum_{m=1}^M  (c_{i_{\sigma(m)}j_m}- c_{i_m j_m})<0,
\]
which is a contradiction.
Thus $\supp\Pi$ is $C$-cyclically monotone.

Conversely,  assume that $\supp\Pi$ is $C$-cyclically monotone.
We define $\xi\in \mathbb{R}^N$ by
\[
\xi_i:= \sup_{M\in \mathbb{N}}\max\left\{  \sum_{m=1}^{M}  \left(c_{i_{m}j_{m}}-c_{i_{m+1}  j_{m}}\right)
             \ \bigg| \ (i_m, j_m)\in  \supp\Pi \ \ \text{for}\ 1\leq m\leq  M,\ \   i_{M+1}=i\right\}.
\]
For $z\in \mathbb{R}^N$, if we define $z^C\in \mathbb{R}^N$ by
\[
z_j^C:=\min_{1\leq i\leq N}\{z_i+c_{ij}\} \quad \text{for}\ 1\leq j\leq N,
\]
then we observe that 
\[
\sum_{j=1}^N\xi^C_jy_j -\sum_{i=1}^N\xi_ix_i
=\sum_{i,j=1}^N \left(\xi^C_j-\xi_i\right) \pi'_{ij}     
\leq \sum_{i,j=1}^N c_{ij} \pi'_{ij}=\langle C, \Pi'\rangle     
\quad \text{for}\ \Pi' \in \Pi(x,y).
\]
For $(i',j')\in  \supp\Pi$, it follows that 
\begin{align*}
\xi_i
&\geq 
 \sup_{M\in \mathbb{N}}\max\left\{  \sum_{m=1}^{M}  \left(c_{i_{m}j_{m}}-c_{i_{m+1}  j_{m}}\right)
             \ \bigg| \begin{array}{l}  (i_m, j_m)\in  \supp\Pi, \ 1\leq m\leq M-1,\\ (i_M, j_M)=(i',j'), \  i_{M+1}=i \end{array}\right\}\\
&= \xi_{i'}+\left(c_{i' j'}-c_{i  j'}\right).            
\end{align*}
Thus we find that 
\[
\xi^C_{j'}=\min_{1\leq i\leq N}\{ \xi_i+c_{i  j'} \} \geq  \xi_{i'}+c_{i' j'},
\]
which means that $\xi^C_{j}-\xi_{i} =c_{ij}$ if $(i,j)\in  \supp\Pi$.
It turns out that 
\begin{align*}
\sum_{j=1}^N\xi^C_jy_j -\sum_{i=1}^N\xi_ix_i
&\leq  \min_{\Pi' \in \Pi(x,y)} \langle C, \Pi'\rangle \\
&\leq  \langle C, \Pi \rangle
 =\sum_{i,j=1}^N \left(\xi^C_j-\xi_i\right) \pi_{ij}   
=\sum_{j=1}^N\xi^C_jy_j -\sum_{i=1}^N\xi_ix_i,
\end{align*}
implying the optimality of $\Pi$.
\end{proof}

By the proof, we notice the following relation. 
\begin{align*}
\mathcal{C}(x,y)
&=
\sup\left\{
\left\langle (-\xi, \eta), (x,y)\right\rangle \ \big|\ (\xi,\eta) \in \mathbb{R}^N\times\mathbb{R}^N ,\ \eta_j-\xi_i\leq  c_{ij}\quad \text{for\ }1\leq i,j \leq N\right\} \\
&=
\sup\left\{
\left\langle (-z, z^C), (x,y)\right\rangle \ \big|\ z \in \mathbb{R}^N \right\}.
\end{align*}
This relation called the \emph{Kantorovich duality}
(see \cite{Vi}*{Chapter~5} for more details).
\begin{corollary}\label{convop}
Let $((x^n, y^n))_{n\in \mathbb{N}}$ be a sequence in $(\mathcal{P}_N \times \mathcal{P}_N, \|\cdot\|_2)$ converging to~$(x,y)$.

For an optimal coupling $\Pi^n\in \Pi(x^n,y^n)$, 
$(\Pi^n)_{n\in \mathbb{N}}$ contains a convergent subsequence and the limit of any convergent subsequence  is an optimal coupling between~$x,y$.
Hence $\mathcal{C}$ is continuous  on $(\mathcal{P}_N \times \mathcal{P}_N , \|\cdot\|_2)$.

Conversely, for an optimal coupling $\Pi_\ast \in \Pi(x,y)$,
there exists  a sequence of optimal couplings $\Pi^n_\ast\in \Pi(x^n,y^n)$  containing a  subsequence converging to~$\Pi_\ast$.
\end{corollary}
\begin{proof}
%
By the compactness of $(\mathcal{P}_{N\times N}, \|\cdot\|_2)$,
there is a convergent subsequence  $(\Pi^{n(l)})_{l\in \mathbb{N}}$ of $(\Pi^n)_{n\in \mathbb{N}}$.
Set $\Pi:=\lim_{l\to \infty} \Pi^{n(l)}$.
We see that 
\begin{equation}\label{marginal}
\sum_{k=1}^N\pi_{ik}=\sum_{k=1}^N \lim_{l \to \infty}\pi^{n(l)}_{ik}=x_i,\quad
\sum_{k=1}^N \pi_{kj}=\sum_{k=1}^N \lim_{l \to \infty}\pi^{n(l)}_{kj}=y_j, \quad \text{for $1\leq i,j\leq N$},
\end{equation}
hence $\Pi \in \Pi(x,y)$.
For  $l\in \mathbb{N}$ large enough,  $\supp \Pi \subset \supp \Pi^{n(l)}$ holds.
By  Proposition~\ref{cmono},  $\supp \Pi^{n(l)}$ is $C$-cyclically monotone 
and so is $\supp \Pi$, hence $\Pi$ is optimal.

To prove the converse implication, 
we choose a subsequence of $((x^n,y^n))_{n\in \mathbb{N}}$, still denoted by $((x^n,y^n))_{n\in \mathbb{N}}$,
such that $(x^n,y^n)\neq (x,y)$ and  $\supp x \otimes y \subset  \supp  x^n \otimes y^n$ for $n\in \mathbb{N}$.
Define $\delta_n\in (0,1)$ and $\widetilde{x}^n, \widetilde{y}^n\in \mathcal{P}_N$ by
\begin{equation}\label{delta}
\begin{split}
&\delta_n:=\min_{(i,j)\in\supp x\otimes y}\left\{ \frac{x_i^n}{x_i}, \frac{y_j^n}{y_j}\right\}\cdot  
\left(1-\sqrt{ \max_{(i,j)\in\supp x^n\otimes y^n}\{ |x_i^n-x_i|, |y_j^n-y_j|\} } \right)\xrightarrow{n\to\infty}1,\\
&\widetilde{x}^n:=\frac{x^n-\delta_n x }{1-\delta_n}\xrightarrow{n\to\infty}x,\qquad
 \widetilde{y}^n:=\frac{y^n-\delta_n y }{1-\delta_n}\xrightarrow{n\to\infty}y,
\end{split}
\end{equation}
respectively.
For an optimal coupling $Q^n \in \Pi( \widetilde{x}^n, \widetilde{y}^n)$,
we denote by $(Q^{n(l)})_{l\in \mathbb{N}}$ and $Q$ a convergent subsequence of  $(Q^n)_{n\in \mathbb{N}}$ and its limit, respectively. 
Then $Q\in \Pi(x,y)$ is optimal.
If we set 
\[
Q_\ast^n:=\delta_n Q +(1-\delta_n) Q^n,\qquad
\Pi_\ast^n:=\delta_n \Pi_\ast +(1-\delta_n) Q^n,
\]
then $Q_\ast^n, \Pi_\ast^n \in \Pi(x^n,y^n)$ and $\lim_{n\to\infty}\Pi_\ast^n=\Pi_\ast$.
It follows that 
\[
\langle C, Q_\ast^n\rangle
=\delta_n \mathcal{C}(x,y)+(1-\delta_n )\mathcal{C}(\widetilde{x}^n, \widetilde{y}^n)
=
\langle C, \Pi_\ast^n\rangle.
\]
For  $l\in \mathbb{N}$ large enough,  $\supp Q \subset \supp Q^{n(l)}$ holds,  hence $\supp Q_\ast^{n(l)} =\supp  Q^{n(l)}$,
which is $C$-cyclically monotone.
Then $Q_\ast^{n(l)}$ is optimal by Proposition~\ref{cmono}, and $\Pi_\ast^{n(l)}$ is also optimal.
This completes the proof of the corollary.
\end{proof}

\begin{remark}
By Corollary~\ref{convop}, a sequence of optimal couplings contains a convergent subsequence,
however  the sequence itself may not converge.
Indeed, if  $C\in M_{N}(\mathbb{R})$ is the all 1s matrix, then $\supp x\otimes y$ is $C$-cyclically monotone for any $x,y \in \mathcal{P}_N$,
consequently  every element in  $\Pi(x,y)$ is optimal by Proposition~\ref{cmono}.
\end{remark}

\section{relaxation}
We introduce a relaxation  of  the optimal transport problem~\eqref{wass} via strictly convex functions.
Throughout this section,
we choose a strictly convex, continuous function $f$ on $[0,1]$  such that $f\in C^1((0,1])$ and fix it unless otherwise indicated.
Then the limit $f'(0):=\lim_{\varepsilon \downarrow 0} f'(\varepsilon)$ exists in $[-\infty,\infty)$
and  $\lim_{\varepsilon \downarrow 0} \varepsilon f'(\varepsilon)=0$
by  the convexity of $f$.
To make sense of the later function  when $r_0=0$,
we employ throughout the convention that  $\pm\infty \cdot 0:=0$.
Define  $D_f: [0,1] \times [0,1] \to (-\infty,\infty]$ by 
\[
D_f(r,r_0)=f(r)-f(r_0)-f'(r_0)(r-r_0).
\]
It follows from the strict convexity of $f$ that $D_f(r,r_0)\in[0,\infty]$  and $D_f(r,r_0)=0$ if and only if $r=r_0$.
The function $D_f$ is continuous on $[0,1] \times (0,1]$ and lower semicontinuous on $[0,1]\times [0,1]$.
Moreover, for $r_0\in (0,1]$, the function $D_f(\cdot, r_0)$ is strictly convex,  continuous on $[0,1]$.
%
\begin{definition}
Fix  $C\in M_{N}(\mathbb{R})$
and a strictly convex, continuous function $f$ on $[0,1]$  such that $f\in C^1((0,1])$.
\begin{enumerate}
\setlength{\leftskip}{-17pt}
\item
We define the \emph{Bregman divergence} associated to $f$ on  $\mathcal{P}_{N \times N}\times \mathcal{P}_{N \times N}$ by 
\[
\mathcal{D}_f(\Pi, \widetilde{\Pi}):= 
\displaystyle \sum_{i,j=1}^N D_f(\pi_{ij},   \widetilde{\pi}_{ij})
\] 
%
\item
For $\gamma\geq 0$ and $x,y\in \mathcal{P}_N$, 
we define a function  $\mathcal{F}^\gamma_{x,y}$  on $\Pi(x,y)$ by 
\[
\mathcal{F}^\gamma_{x,y}(\Pi) :=\langle C, \Pi\rangle+\gamma  \mathcal{D}_f(\Pi, x\otimes y).
\]
We call $P^\gamma \in \Pi(x,y)$ an $(f,\gamma)$-coupling between $x,y$ if 
\[
\mathcal{F}^\gamma_{x,y}(P^\gamma)=\inf_{\Pi\in \Pi(x,y)}\mathcal{F}^\gamma_{x,y}(\Pi).
\]
\item
For $\gamma \geq 0$ and $\lambda\geq 0$, we define two functions $\mathcal{C}^\gamma$ and  $\mathcal{C}_\lambda$  on $\mathcal{P}_N \times\mathcal{P}_N$ by 
\begin{align*}
\mathcal{C}^\gamma(x,y)& :=\inf\left\{ \langle C, P^\gamma\rangle \bigm|
\text{$P^\gamma \in \Pi(x,y)$ is  an $(f,\gamma)$-coupling}
\right\},\\
\mathcal{C}_\lambda(x,y)&:= \inf\{  \langle C, \Pi\rangle \bigm| \ \Pi \in \Pi(x,y),\ \mathcal{D}_f(\Pi, x\otimes y)  \leq \lambda \},
\end{align*}
respectively.
\item
We define a function $\Lambda$ on $\mathcal{P}_N \times\mathcal{P}_N$ by 
\[
\Lambda(x,y):=\inf\left\{ \mathcal{D}_f(\Pi, x\otimes y) \bigm| \Pi\in \Pi(x,y)  \text{\ is an optimal coupling}\right\}.
\]
\end{enumerate}
\end{definition}
\begin{remark}
A relaxed optimal transport problem via Bregman divergences is also studied by
{Dessein}, {Papadakis} and {Rouas}~\cite{DPR},
where a Bregman divergence is determined by  a convex function $\phi$  on $M_{N}(\mathbb{R})$
and all component of $C$ are nonnegative.
They used the Bregman projection
and  chose the projection point of the minimizer of $\phi$ instead of $x\otimes y$ as a base point.
The choice of the base point is crucial for Bregman divergences other than the Kullback--Leibler divergence.
In our method, the Bregman projection does not appear and the assumption of convex functions is milder  than theirs.
For example, the function  $f(r):=-r^{1/2}e^{-r}-r$  can be treated in our setting, 
but  $\phi(\Pi):=\sum_{i,j=1}^N f(\pi_{ij})$ cannot be treated in the setting in~\cite{DPR}.
\end{remark}

It is trivial that 
an $(f,0)$-coupling is an optimal coupling and  $\mathcal{C}^0=\mathcal{C}$.
For $x,y\in \mathcal{P}_N$ and $\Pi\in \Pi(x,y)$, we see that 
\[
\mathcal{D}_f(\Pi, x\otimes y)=\sum_{(i,j)\in \supp x\otimes y} D_f(\pi_{ij},  x_iy_j) \in[0,\infty),
\]
and $\mathcal{D}_f(\Pi, x\otimes y)=0$ if and only if $\Pi=x\otimes y$.
By the strict convexity and the continuity of $\mathcal{D}_f(\cdot, x\otimes y)$ on the compact set  $(\Pi(x,y), \|\cdot\|_2)$, we have the following.
\begin{proposition}
For $x,y \in \mathcal{P}_N$ and $\gamma>0$,  an $(f,\gamma)$-coupling between $x,y$ is uniquely determined.
Moreover, there exists a unique optimal coupling $\Pi_\ast \in \Pi(x,y)$ such that $\mathcal{D}_f(\Pi_\ast , x\otimes y)=\Lambda(x,y)$.
\end{proposition}

Computing $\mathcal{C}_\lambda(x,y)$ is equivalent to solving the variational problem~\eqref{wass} under the constraint $\mathcal{D}_f(\Pi, x\otimes y)=\lambda$, and reduces to  computing $\mathcal{C}^\gamma(x,y)$.
In the case $f(r)=r\log r$, this property is mentioned (but not proved)  in~\cite{Cuturi}.
See also \cite{DPR}*{Theorem~15}.
\begin{theorem}\label{inc}
Given $x,y\in \mathcal{P}_N$ and  $\lambda \in [0, \Lambda(x,y)]$, it follows that
 \begin{align}\label{bdd}
\mathcal{C}_{\lambda} (x,y)
= \inf\left\{  \langle C, \Pi\rangle \bigm|  \Pi \in \Pi(x,y),\ \mathcal{D}_f(\Pi, x\otimes y) =\lambda\right\}.
\end{align}
Furthermore, if $\lambda=\mathcal{D}_f(P^\gamma,  x\otimes y )$ holds for some $\gamma \geq0$,  where $P^\gamma$ is  an $(f,\gamma)$-coupling between~$x,y$, 
then $\mathcal{C}_{\lambda} (x,y)=\mathcal{C}^\gamma(x,y)$.
\end{theorem}
\begin{proof}
If  $\lambda\in\{ 0,\Lambda(x,y)\}$, then \eqref{bdd} trivially holds.
Assume $\lambda\in (0,\Lambda(x,y)) \neq \emptyset$.
By the continuity of $\mathcal{D}_f(\cdot,  x\otimes y )$ on $\Pi(x,y)$, 
the set 
\[
\{\Pi \in \Pi(x,y)\ |\  \mathcal{D}_f(\Pi, x\otimes y) \leq \lambda\}
\]
is compact in $(\Pi(x,y), \|\cdot\|_2)$.
Then  there exists $\Pi_\lambda\in \Pi(x,y)$ such that
\[
\mathcal{D}_f(\Pi_\lambda, x\otimes y)\leq \lambda,
\quad
\mathcal{C}_{\lambda}(x,y)= \langle C, \Pi_\lambda \rangle>\mathcal{C}(x,y).
\]
If $\mathcal{D}_f(\Pi_\lambda , x\otimes y)<\lambda$, then there exists  $t\in (0,1)$ such that 
\[
\mathcal{D}_f(  (1-t) \Pi_\lambda+t \Pi_\ast , x\otimes y ) =\lambda
\]
by  the intermediate value theorem,
where $\Pi_\ast \in \Pi(x,y)$ is an optimal coupling such that $\mathcal{D}_f(\Pi_\ast , x\otimes y)=\Lambda(x,y)$.
However,
\[
\mathcal{C}_\lambda(x,y)
\leq \langle C,  (1-t) \Pi_\lambda+t \Pi_\ast\rangle
=(1-t)\mathcal{C}_\lambda(x,y)+t\mathcal{C}(x,y)<\mathcal{C}_\lambda(x,y),
\]
which is a contradiction.
Thus $\mathcal{D}_f(\Pi_\lambda , x\otimes y)=\lambda$ and  \eqref{bdd} holds.
In addition,  in the case  of $\lambda=\mathcal{D}_f(P^\gamma,  x\otimes y )$, 
we see that
\begin{align*}
\mathcal{F}^\gamma_{x,y}(\Pi_\lambda)
=\mathcal{C}_\lambda(x,y) +\gamma \lambda 
\leq \langle C, P^\gamma\rangle+\gamma \lambda
=
\mathcal{F}^\gamma_{x,y}(P^\gamma) 
\leq  \mathcal{F}^\gamma_{x,y}(\Pi_\lambda),
\end{align*}
which implies $\mathcal{C}_{\lambda} (x,y)=\mathcal{C}^\gamma(x,y)$.
\end{proof}
By Theorems~\ref{dconti}, \ref{mono} and  Proposition~\ref{second} below, 
for $x,y\in \mathcal{P}_N$, 
if  $\supp x\otimes y$ is not $C$-cyclically monotone and $f'(0)=-\infty$, then 
\[
(0,\Lambda(x,y))
=\left\{ \mathcal{D}_f(P^\gamma, x\otimes y ) \ |\ \text{$\gamma>0$ and $P^\gamma \in \Pi(x,y)$ is an $(f,\gamma)$-coupling}\right\}.
\]

We prove the continuity of $\Lambda$.
\begin{proposition}\label{lambdaconti}
The function $\Lambda$ is lower semicontinuous on $\mathcal{P}_N \times \mathcal{P}_N$
and continuous on the interior of $\mathcal{P}_N \times \mathcal{P}_N$.
Moreover, if either $f'(0)\in \mathbb{R}$ or $f(r)=r\log r$, 
then  $\Lambda$ is continuous on $\mathcal{P}_N \times \mathcal{P}_N$.
\end{proposition}
\begin{proof}
Let $((x^n, y^n))_{n\in \mathbb{N}}$ be a sequence in $(\mathcal{P}_N \times \mathcal{P}_N, \|\cdot\|_2)$ converging to $(x,y)$.
Set 
\[
S':=\supp  x^n \otimes y^n \setminus \supp x \otimes y,
\]
which is empty  if $(x,y)$ is an interior point of $\mathcal{P}_N \times \mathcal{P}_N$.
Passing to subsequences if necessary, 
we may assume that
$\supp x \otimes y \subset  \supp  x^n \otimes y^n$ 
and $\supp  x^n \otimes y^n$ does not depend on $n\in \mathbb{N}$
without loss of generally. 
Let  $\Pi_\ast^n \in \Pi(x^n,y^n)$ be an optimal coupling such that  $\mathcal{D}_f(\Pi_\ast^n , x^n\otimes y^n)=\Lambda(x^n,y^n)$.
Since any convergent  subsequence of $(\Pi_\ast^n)_{n\in \mathbb{N}}$  converges to an optimal coupling between $x,y$ by Corollary \ref{convop},
we have
\[
\Lambda(x,y)\leq \liminf_{n\to \infty}  \Lambda(x^n,y^n),
\]
which implies the lower semicontinuity of $\Lambda$ on $\mathcal{P}_N \times \mathcal{P}_N$.
Again by Corollary \ref{convop}, there is a sequence of optimal couplings $\Pi^n \in \Pi(x^n,y^n)$ such that 
$\lim_{l\to \infty}\Pi^{n(l)}= \Pi_\ast$ for some $(n(l))_{l\in\mathbb{N}}\subset \mathbb{N}$.
If $(i,j)\in \supp x\otimes y$, then
\[
\lim_{l\to \infty} f'( x^{n(l)}_iy^{n(l)}_j   )\pi^{n(l)}_{ij}
=f'(x_iy_j) \pi_\ast{}_{ij}
\]
by $f\in  C^1((0,1])$.
On the other hand,  
since   $\pi^{n(l)}_{ij}  \leq \min \{ x^{n(l)}_i, y^{n(l)}_j \}$ holds for $1\leq i, j\leq N$,
if either $f'(0)\in \mathbb{R}$ or $f(r)=r\log r$, then 
\[
\lim_{l\to \infty} \left| f'( x^{n(l)}_iy^{n(l)}_j   )\pi^{n(l)}_{ij} \right|
\leq \lim_{l\to \infty} \left| f'( x^{n(l)}_iy^{n(l)}_j   )\right| \cdot \min\{ x^{n(l)}_i, y^{n(l)}_j\}=0
\quad
\text{for\ }(i,j)\in S'.
\]
Thus
if $f'(0)\in \mathbb{R}$, $f(r)=r\log r$, 
or $(x,y)$  is an interior point of $\mathcal{P}_N \times \mathcal{P}_N$, 
then 
\[
\limsup_{n\to \infty}  \Lambda(x^n,y^n)
\leq \lim_{l\to \infty}  \mathcal{D}_f(\Pi^{n(l)}, x^{n(l)}\otimes y^{n(l)} )
=\mathcal{D}_f(\Pi_\ast, x\otimes y)=\Lambda(x,y).
\]
This completes the proof of the proposition.
\end{proof}
\begin{remark}
In Proposition~\ref{lambdaconti},
if  $f'(0)=-\infty$, 
then $\Lambda$ is  not necessarily continuous on $\mathcal{P}_N\times \mathcal{P}_N$.
Indeed, let $f(r)=-r^{p}/p$ for $p\in (0,1/2)$ and $C=(\delta_{ij})_{1\leq i,j \leq 2}\in M_2(\mathbb{R})$.
We choose  $x^a, y^a \in \mathcal{P}_2$ and $\Pi^{a,s}\in \mathcal{P}_{2\times 2}$ as 
\begin{align*}
&x_1^a=a, \quad x_2^a=1-a, \qquad
y_1^a=1-a, \quad
y_2^a=a,\quad
\text{where}\ a\in [0,1/2],\\
&\pi_{11}^{a,s}=s,\quad \pi_{12}^{a,s}=a-s, \quad \pi^{a,s}_{21}=1-a-s  \quad\pi_{22}^s=s,
\quad
\text{where}\ s\in [0,a],
\end{align*}
respectively.
Then $\Pi^{a,0}$ is an unique optimal coupling. 
We see that 
\begin{align*}
\Lambda(x^0,y^0)
=0,
\qquad
\lim_{a\downarrow 0}\Lambda(x^a,y^a)
=-\lim_{a\downarrow 0} f'(a^2)a=\lim_{a\downarrow 0}a^{2p-1}=\infty.
\end{align*}
\end{remark}

We consider a condition such that $(f,\gamma)$-couplings lie in the interior of $\Pi(x,y)$.
\begin{lemma}\label{int}
For $x, y\in \mathcal{P}_N$ and $\gamma>0$,  let $P^\gamma\in \Pi(x,y)$ be an $(f,\gamma)$-coupling.
Assume that $f'(0)=-\infty$.
Then
\[
  \supp P^\gamma=\supp x\otimes y
\]
and  there exists $(\alpha, \beta)\in \mathbb{R}^{\supp x\otimes y}$ such that 
\begin{equation}\label{lag}
c_{ij}+\gamma \left(f'(p^\gamma_{ij})-f'(x_iy_j)\right)=\alpha_i+\beta_j
\qquad \text{for\ }(i,j)\in \supp x\otimes y.
\end{equation}
\end{lemma}
\begin{proof}
It follows form  Lemma~\ref{supp} that 
\[
\supp P^\gamma \subset \supp x\otimes y=:S.
\]
Assume that $\supp P^\gamma \neq S$.
For $t\in (0,1)$, set 
\[
\Pi^t:=(1-t)  P^\gamma+t x\otimes y\in \Pi(x,y).
\]
Since $P^\gamma$ is a unique minimizer of $\mathcal{F}^\gamma_{x,y}$ on $\Pi(x,y)$, it turns out that 
\begin{align*}
0\leq \lim_{t \downarrow 0} \frac{d}{dt}\mathcal{F}^\gamma_{x,y}(\Pi^t)
&= \sum_{i,j\in \supp P^\gamma}  (x_iy_j -p^\gamma_{ij})\left\{ c_{ij}+\gamma \left(f'( p_{ij}^\gamma )-f'(x_iy_j)\right) \right\}  \\
&\qquad+\sum_{i,j\in  S\setminus \supp P^\gamma}  x_iy_j
\left\{ c_{ij}+\gamma \left( \lim_{t \downarrow 0} f'(t x_iy_j)-f'(x_iy_j)\right) \right\}  \\
&=-\infty.
\end{align*}
This is a contradiction, hence $\supp\Pi^\gamma=\supp x\otimes y$.

Assume $|\supp x|=|\supp y|=N$.
If we regard $\mathcal{P}_{N\times N}$ as a subset of $\mathbb{R}^{N\times N}$,
then $P^\gamma$ lies in an open set $(\mathbb{R}_{>0})^{N\times N}$.
For $1\leq i,j \leq N$, define $X_i, Y_j:(\mathbb{R}_{>0})^{N\times N}\to \mathbb{R}$ by
\[
X_i(Z):=\sum_{k=1}^N z_{ik}, \qquad
Y_j(Z):=\sum_{k=1}^N z_{kj}.
\]
We see that, for $Z\in (\mathbb{R}_{>0})^{N\times N}$, 
$Z\in \Pi(x,y)$ if and only if
\[
X_i(Z)=x_i\quad \text{for\ }1\leq i\leq N, \qquad
Y_j(Z)=y_j\quad \text{for\ }1\leq j\leq N-1.
\]
Since  the Jacobian matrix of $(X_1, \cdots, X_N, Y_1, \cdots, Y_{N-1})$ has always rank $2N-1$, 
the method of Lagrange multipliers yields that there is $(\alpha, \beta)\in \mathbb{R}^{N \times N}$ such that $\beta_N=0$ and 
\begin{align*}
c_{ij}+\gamma \left(f'(p^\gamma_{ij})-f'(x_iy_j)\right)
&=\nabla_{z_{ij}} \mathcal{F}^{\gamma}_{x,y}(P^\gamma)\\
&= \sum_{k=1}^N \left( \alpha_k \nabla_{z_{ij}} X_k(P^\gamma)+ \beta_k \nabla_{z_{ij}} Y_k(P^\gamma) \right)
=\alpha_{i}+\beta_j.
\end{align*}
The case of either $|\supp x|\neq N$ or  $|\supp y|\neq N$ is proved analogously.
\end{proof}
\begin{remark}\label{interior}
To apply the method of Lagrange multipliers, an extremum should be attained  in the interior of the constraint set  as in the proof of Lemma~\ref{int}.
%
Note that the number of the constraint conditions is $|\supp x|+|\supp  y|-1$, not  $|\supp x|+|\supp  y|$.
\end{remark}
\begin{remark}\label{38}
In Lemma~\ref{int}, the assumption  $f'(0)=-\infty$ is necessary.
Indeed,  if  $f'(0) \in \mathbb{R}$ then $f\in C^1([0,1])$.
We choose  $x, y^a \in \mathcal{P}_2$ and $P^{a,s}\in \mathcal{P}_{2\times 2}$ as 
\begin{align*}
&x_1=x_2=\frac12, \qquad
y_1^a=1-a, \quad
y_2^a=a,\quad
\text{where}\ a\in [0,1/2],\\
&p_{11}^{a,s}=\frac12-s,\quad p_{12}^{a,s}=s, \quad p_{21}^{a,s}=\frac12-a+s, \quad p_{22}^{a,s}=a-s,
\quad
\text{where}\ s\in [0,a],
\end{align*}
respectively.
Then $P(x,y^a)=\{P^{a,s}\}_{s\in  [0,a]}$.  
We see that 
\[
M_a:=\sup_{s \in [0,a]} \big| -f'(p_{11}^{a,s})+f'(p_{12}^{a,s})+f'(p_{21}^{a,s})-f'(p_{22}^{a,s}) \big|+1
\]
is finite  by the continuity of $f'$ on $[0,1]$.
If we take  $C=(\delta_{ij})_{1\leq i,j \leq 2}\in M_2(\mathbb{R})$ and $\gamma\in (0, 2/M_a)$,
then, for $s\in (0,a)$ with $a\neq 0$,  it turns out that 
\begin{align*}
\frac{\partial}{\partial s}\mathcal{F}^\gamma_{x,y}(P^{a,s})
=-2+\gamma \left(-f'(p_{11}^{a,s})+f'(p_{12}^{a,s})+f'(p_{21}^{a,s})-f'(p_{22}^{a,s}) \right) 
\leq -2+\gamma M_a<0,
\end{align*}
implying that  $P^{a,a}$ is an $(f,\gamma)$-coupling, where  $\supp P^{a,a}\neq \supp x\otimes y^a$.
\end{remark}
If we  define $\widetilde{f}:[0,1]\to \mathbb{R}$ by 
\[
\widetilde{f}(r):=f(r)-f(0)-r(f(1)-f(0)),
\]
then $\widetilde{f}$ is again strictly convex, continuous on $[0,1]$ and $\widetilde{f}\in C^1((0,1])$.
We see that $\widetilde{f}(0)=\widetilde{f}(1)=0$.
Moreover, $D_{\widetilde{f}}=D_f$ on  $[0,1] \times [0,1]$.
Thus, when we discuss the relaxed optimal transport problem  via the Bregman divergences associated to $f$,
we may assume $f(0)=f(1)=0$ without loss of generality.
In this case,  the convexity of $f$ yields that $f \leq 0$ on $[0,1]$.
We estimate the difference between $\mathcal{C}$ and $\mathcal{C}^\gamma$.
\begin{proposition}
Given $x,y \in \mathcal{P}_N$ and $\gamma>0$, 
let $P^\gamma\in \Pi(x,y)$ be an $(f,\gamma)$-coupling.  
Then 
\begin{align}\label{errors}
\begin{split}
\frac{1}{\gamma}\left(\mathcal{C}^\gamma(x,y)-\mathcal{C}(x,y)\right)
&\leq \Lambda(x,y)- \mathcal{D}_f(P^\gamma, x\otimes y)\leq \Lambda(x,y).
\end{split}
\end{align}
In the case of $f'(0)=-\infty$, 
\[
\frac{1}{\gamma}\left(\mathcal{C}^\gamma(x,y)-\mathcal{C}(x,y)\right)
=-\mathcal{D}_f(\Pi_\ast, P^\gamma  )
+ \Lambda(x,y)
-\mathcal{D}_f(P^\gamma, x\otimes y ),
\]
where $\Pi_\ast\in \Pi(x,y)$ is an optimal coupling such that $\Lambda(x,y)=\mathcal{D}_f(\Pi_\ast, x\otimes y)$.
On the other hands, if  $f(0)=f(1)=0$, then 
\[
\Lambda(x,y)
\leq
-f'\left(\min_{(i,j)\in \supp x\otimes y}x_iy_j\right)
+
\sum_{(i,j)\in \supp x\otimes y}\left( -f(x_iy_j)+f'(x_iy_j)x_iy_j \right).
\]
\end{proposition}
\begin{proof}
It follows from  the definition of $\mathcal{F}_{x,y}^\gamma$ with the nonnegativity of $\mathcal{D}_f$ that 
\begin{align*}
 \mathcal{C}^\gamma(x,y)
 &\leq
 \mathcal{C}^\gamma(x,y)+\gamma \mathcal{D}_f(P^\gamma, x\otimes y)
= \mathcal{F}_{x,y}^\gamma(P^\gamma)
\leq \mathcal{F}_{x,y}^\gamma(\Pi_\ast)=\mathcal{C}(x,y)+\gamma \Lambda(x,y),
\end{align*}
implying the inequalities in \eqref{errors}.
Assume $f'(0)= -\infty$.
By Lemma~\ref{int},  there exists $(\alpha, \beta)\in \mathbb{R}^{\supp x\otimes y}$ satisfying~\eqref{lag}.
Since we have
\[
\sum_{(i,j)\in \supp x\otimes y} (\alpha_i +\beta_j)\pi_{ij}
=\sum_{i\in \supp x} \alpha_i x_i +\sum_{j\in \supp y} \beta_j y_j
\quad \text{for\ } \Pi\in \Pi(x,y),
\]
we multiply \eqref{lag} by $\pi_\ast{}_{ij}-p^\gamma_{ij}$ and  take the sum over $(i,j)\in \supp x\otimes y$ to have
\begin{align*}
\frac{1}{\gamma}\left(\mathcal{C}^\gamma(x,y)-\mathcal{C}(x,y)\right)
&=\sum_{(i,j)\in \supp x\otimes y}
(\pi_\ast{}_{ij}-p^\gamma_{ij})\left(f'(p^\gamma_{ij})-f'(x_iy_j)\right)\\
&=
-\mathcal{D}_f(\Pi_\ast, P^\gamma)+\mathcal{D}_f(\Pi_\ast, x\otimes y)-\mathcal{D}_f(P^\gamma, x\otimes y).
\end{align*}

In the case of  $f(0)=f(1)=0$,  the convexity of $D_f(\cdot, r_0)$ on $[0,1]$ yields that 
\begin{align*}
\Lambda(x,y)
&=\mathcal{D}_f(\Pi_\ast, x\otimes y)
=\sum_{(i,j)\in \supp x\otimes y} D_f(\pi_\ast{}_{ij}, x_iy_j )\\
&\leq
\sum_{(i,j)\in \supp x\otimes y}\left\{
(1-\pi_\ast{}_{ij})  D_f(0, x_iy_j )+\pi_\ast{}_{ij}  D_f(1, x_iy_j )\right\}\\
&=
\sum_{(i,j)\in \supp x\otimes y}\left( -f(x_iy_j)-f'(x_iy_j)(\pi_\ast{}_{ij}- x_iy_j)\right)\\
&\leq 
-f'\left(\min_{(i,j)\in \supp x\otimes y}x_iy_j\right)
+
\sum_{(i,j)\in \supp x\otimes y}\left( -f(x_iy_j)+f'(x_iy_j)x_iy_j \right),
\end{align*}
where  the last inequality follows from the monotonicity of $f'$ on $(0,1]$.
This completes the proof of the proposition.
\end{proof}
We prove the continuity of $(f,\gamma)$-couplings.
The limit of $(f,\gamma)$-coupling between two given points in $\mathcal{P}_N$ 
as $\gamma\downarrow  0$ and $\gamma\uparrow \infty$ are discussed in \cite{DPR}*{Properties~7--9} for Bregman divergences,
and in \cite{PC}*{Proposition~4.1} for the Kullback--Leibler divergences.
In Euclidean setting, 
$\mathcal{C}^{\gamma}$ is  lower semicontinuous, but not continuous in general (see \cite{CDPS}*{Lemma~2.4}).
%
\begin{theorem}\label{dconti}
Let $((x^n, y^n))_{n\in \mathbb{N}}$ be  a  sequence in $(\mathcal{P}_N \times \mathcal{P}_N, \|\cdot\|_2)$ converging to $(x,y)$
and $(\gamma^n)_{n\in \mathbb{N}}$ a sequence  in $(0,\infty)$ converging to $\gamma\in[0,\infty]$, respectively.
For an $(f,\gamma_n)$-coupling $P^{\gamma_n}$ between $x^n,y^n$,
\begin{align*}
\lim_{n\to \infty} P^{\gamma_n}=
\begin{cases}
x\otimes y &\gamma=\infty,\\
P^{\gamma} &\gamma\in (0,\infty),\\
\Pi_\ast &\gamma=0,
\end{cases}
\end{align*}
where $P^{\gamma}$ is an $(f,\gamma)$-coupling between $x,y$,
and $\Pi_\ast$ is  an optimal coupling between $x,y$ such that  $\mathcal{D}_f(\Pi_\ast , x\otimes y)=\Lambda(x,y)$.
Hence the function on $(0,\infty) \times \mathcal{P}_N \times \mathcal{P}_N$
sending $(\gamma, x,y)$ to $\mathcal{C}^{\gamma}(x,y)$  is continuous 
and can be extended to $[0,\infty] \times \mathcal{P}_N \times \mathcal{P}_N$ continuously.
\end{theorem}
\begin{proof}
By the compactness of  $(\mathcal{P}_{N\times N}, \|\cdot\|_2)$, $(P^{\gamma_{n}})_{n\in \mathbb{N}}$ contains a convergent subsequence,
still denoted by $(P^{\gamma_{n}})_{n\in \mathbb{N}}$.
An argument similar to that in~\eqref{marginal} yields that  $\Pi:=\lim_{n\to \infty} P^{\gamma_n} \in \Pi(x,y)$.
We may assume that
$\supp x \otimes y \subset  \supp  x^n \otimes y^n$ 
and $\supp  x^n \otimes y^n$ does not depend on $n\in \mathbb{N}$
without loss of generally. 
Set 
\[
S':= \supp x^n \otimes y^n \setminus \supp x\otimes y
\]
Then $\pi_{ij}=\lim_{n\to \infty}p^{\gamma_n}_{ij}=0$ for $(i,j)\in S'$.
%
\begin{claim}\label{claim}
Assume one of $S'=\emptyset$,   $f'(0)\in \mathbb{R}$  or $\gamma\neq0$.
Then
\[
\lim_{n\to \infty}  \mathcal{D}_f(P^{\gamma_{n}},x^{n}\otimes y^{n}) 
=\mathcal{D}_f(\Pi,x\otimes y).
\]
\end{claim}
\noindent
{\it Proof.}
Recall that 
\begin{align*}
\mathcal{D}_f(P^{\gamma_{n}},x^{n}\otimes y^{n})
=\sum_{(i,j)\in \supp x^n \otimes y^n}  
f(p^{\gamma_{n}}_{ij})-f(x^{n}_i y^{n}_j)-f'(x^{n}_i y^{n}_j)(p^{\gamma_{n}}_{ij}-x^{n}_i  y^{n}_j).
\end{align*}
Since $f\in C([0,1])\cap C^1((0,1])$ and $\lim_{\varepsilon\downarrow 0} \varepsilon f'(\varepsilon)=0$, 
it is enough to show that
\[
\lim_{n\to \infty} 
f'(x_i^{n}y^{n}_j)p^{\gamma_{n}}_{ij}=0
\quad \text{for\ }(i,j)\in S'.
\]
This trivially holds if either $S'=\emptyset$ or $f'(0)\in \mathbb{R}$.

Assume $S'\neq\emptyset$ and  $f'(0)=-\infty$.
Lemma~\ref{int} yields that  $\supp P^{\gamma_{n}}=\supp x^{n} \otimes y^{n}$.
Fix  $(i_1,j_1)\in S'$ and $(i_2,j_2)\in \supp \Pi$.
Then $f'(x_{i_1}^{n}y^{n}_{j_1})p^{\gamma_{n}}_{i_1j_1} \leq 0$ for $n\in \mathbb{N}$ large enough
and
\[
 \lim_{n\to \infty}  f'(p^{\gamma_n}_{i_2j_2} ) p^{\gamma_n}_{i_1j_1} =
 \lim_{n\to \infty}  f'(x_{i_2}^n y_{j_2}^n) p^{\gamma_n}_{i_1j_1} =0.
\] 
By  $\lim_{\varepsilon\downarrow 0} \varepsilon f'(\varepsilon)=0$,
$\lim_{n\to \infty}  f'(p^{\gamma_n}_{i_1j_1} ) p^{\gamma_n}_{i_1j_1} =0$.
In addition,  it follows from  $p^{\gamma_n}_{i_1j_1} \leq x_{i_1}^n$ that 
\[
\lim_{n\to \infty} \left| f'(x_{i_1}^n y_{j_2}^n)  p^{\gamma_n}_{i_1j_1}\right|
\leq \lim_{n\to \infty} \left| f'(x_{i_1}^n y_{j_2}^n) \right|\cdot x_{i_1}^n y_{j_2}^n \cdot \frac{1}{y_{j_2}^n} =0.
\]
Similarly, $\lim_{n\to \infty}f'(x_{i_2}^n y_{j_1}^n) p^{\gamma_n}_{i_1j_1}  =0$ holds.
For $(i,j)=(i_1, j_2), (i_2, j_1)$, 
if $(i, j)\in \supp \Pi$,  $\lim_{n\to \infty} f'( p^{\gamma_n}_{ij}) p^{\gamma_n}_{i_1j_1} =0$ trivially holds.
In the case of  $(i, j)\notin \supp \Pi$, then $ f'( p^{\gamma_n}_{ij}) p^{\gamma_n}_{i_1j_1} \leq 0$ for $n\in \mathbb{N}$ large enough.
For  $\varepsilon\in \mathbb{R}$  small enough, if we set $\Pi^{\varepsilon}\in M_N(\mathbb{R})$ by 
\[
\pi^\varepsilon_{ij}
:=\begin{cases}
p^{\gamma_n}_{ij}+\mathrm{sgn}(\sigma)\varepsilon   &\text{if\ }(i,j)= (i_{\sigma(m)},j_m) \text{\ with $m=1,2$ and $\sigma\in \mathfrak{S}_2$} ,\\
p^{\gamma_n}_{ij} &\text{otherwise},
\end{cases}
\]
then $\Pi^{\varepsilon }\in \Pi(x^n, y^n)$ and 
\begin{align*}
0=
\frac{d}{d\varepsilon}
\mathcal{F}^\gamma_{x^n,y^n}(\Pi^{\varepsilon})\bigg|_{\varepsilon=0}
=\!\sum_{m\in \{1,2\}, \sigma\in \mathfrak{S}_2} 
\mathrm{sgn}(\sigma)\left\{ c_{i_{\sigma(m)},j_m} +\gamma_n \left( f'(p^{\gamma_n}_{i_{\sigma(m)},j_m})-f'(x_{i_{\sigma(m)}}^ny_{j_m}^n)\right) \right\}.
\end{align*}
If $\gamma\neq 0$, then 
\begin{align*}
0&=-\limsup_{n\to \infty}
p^{\gamma_n}_{i_1j_1}
\sum_{m=1,2, \sigma\in \mathfrak{S}_2} 
\mathrm{sgn}(\sigma)\left\{ \frac{c_{i_{\sigma(m)},j_m}}{\gamma_n} + \left( f'(p^{\gamma_n}_{i_{\sigma(m)},j_m})-f'(x_{i_{\sigma(m)}}^ny_{j_m}^n)\right) \right\}\\
&=
\liminf_{n\to \infty} p^{\gamma_n}_{i_1j_1}\left( f'(x_{i_1}^n y_{j_1}^n) +f'(p^{\gamma_n}_{i_1j_2}) +f'(p^{\gamma_n}_{i_2j_1}) \right)\\
&
\leq  \liminf_{n\to \infty} p^{\gamma_n}_{i_1j_1} f'(x_{i_1}^n y_{j_1}^n) 
\leq  \limsup_{n\to \infty} p^{\gamma_n}_{i_1j_1} f'(x_{i_1}^n y_{j_1}^n) 
\leq 0,
\end{align*}
implying $\lim_{n\to \infty}  f'(x_{i_1}^n y_{j_1}^n) p^{\gamma_n}_{i_1j_1}=0$.
This completes the proof of the claim.
$\hfill \diamondsuit$

We see that 
\begin{align*}
\frac{1}{\gamma_n}  \langle  C, P^\gamma \rangle+\mathcal{D}_f(P^{\gamma_n},x^n\otimes y^n) 
&=
\frac{1}{\gamma_n}
\mathcal{F}^{\gamma_n}_{x^n,y^n}(P^{\gamma_n}) \\
&\leq  
 \frac{1}{\gamma_n} 
 \mathcal{F}^\gamma_{x^n,y^n}(x^n\otimes y^n)
=\frac{1}{\gamma_n} \langle  C, x^n\otimes y^n\rangle.
\end{align*}
Since $|\langle  C, \Pi\rangle|$ is uniformly bounded in $\Pi \in \mathcal{P}_{N\times N}$ by $\max_{1\leq i, j\leq N } |c_{ij}|$,
in the case of $\gamma=\infty$,
letting $n\to \infty$ and applying Claim \ref{claim} lead to
$\mathcal{D}_f(\Pi, x\otimes y)\leq 0$, that is, $\Pi=x\otimes y$.

Assume one of $S'=\emptyset$ or $f'(0)\in \mathbb{R}$.
Then, similarly to the proof  of Claim~\ref{claim},
\[
\lim_{n\to \infty}\mathcal{D}_f (\Pi^n, x^n\otimes y^n)
=\mathcal{D}_f (\lim_{n\to \infty}\Pi^n, x\otimes y)
\]
holds for any convergent sequence of $\Pi^n \in \Pi(x^n,y^n)$.
Let $\delta_n\in(0,1)$ and $\widetilde{x}^n,\widetilde{y}^n \in \mathcal{P}_N$ as in \eqref{delta}.
For $\gamma\in (0,\infty)$, it follows from the definition of $P^\gamma$ that 
\begin{align*} 
\mathcal{F}^\gamma_{x,y}(P^\gamma)
\leq
\mathcal{F}^\gamma_{x,y}(\Pi)
&=
\lim_{n\to\infty}\mathcal{F}^{\gamma_n}_{x^n,y^n}( P^{\gamma_n}) \\
&\leq 
\lim_{n\to\infty} \mathcal{F}^{\gamma_n}_{x^n,y^n}(\delta_n P^\gamma +(1-\delta_n) \widetilde{x}^n \otimes \widetilde{y}^n)
=
\mathcal{F}^\gamma_{x,y}(P^\gamma).
\end{align*}
By the uniqueness  of an $(f,\gamma)$-coupling,  we find that $\Pi=P^\gamma$.
On the other hand, there exists a sequence of optimal couplings $\Pi^n \in \Pi(x^n,y^n)$ such that 
$\lim_{l\to \infty}\Pi^{n(l)}= \Pi_\ast$ for some $(n(l))_{l\in\mathbb{N}}\subset \mathbb{N}$ by Corollary~\ref{convop}. 
In the case of  $\gamma=0$,  it turns out that 
\begin{align*} 
\mathcal{C}(x,y)\leq \langle C, \Pi\rangle
&=
\lim_{n\to\infty}\mathcal{F}^{\gamma_n}_{x^n,y^n}( P^{\gamma_n})
\leq 
\lim_{n\to\infty} \mathcal{F}^{\gamma_n}_{x^n,y^n}(\Pi^n)
=\langle C, \Pi_\ast\rangle=
\mathcal{C}(x,y), \\
\mathcal{D}_f(\Pi, x\otimes y)-\Lambda(x,y)
&=\lim_{l\to \infty} \left(
\mathcal{D}_f(P^{\gamma_{n(l)}},x^{n(l)}\otimes y^{n(l)}) -  
\mathcal{D}_f(\Pi^{n(l)}, x^{n(l)}\otimes y^{n(l)})\right)\\
&\leq  
\limsup_{l\to \infty}
\frac{1}{\gamma_{n(l)}}\left(\langle  C, \Pi^{n(l)}\rangle-\langle C, P^{\gamma_{n(l)}}\rangle \right)
\leq0.
\end{align*}
Thus $\Pi \in \Pi(x,y)$ is optimal and  $\mathcal{D}_f(\Pi, x\otimes y)\leq \Lambda(x,y)$, that is, $\Pi=\Pi_\ast$.

It remains to prove the case $\gamma\in [0,\infty)$ under the conditions $S'\neq \emptyset$ and $f'(0)=-\infty$.
In this case,  it follows from Lemma~\ref{int} that  there exists $(\alpha^n, \beta^n)\in \mathbb{R}^{\supp x^n\otimes y^n}$ such that 
\begin{equation}\label{lagn}
c_{ij}+\gamma_n \left(f'(p^{\gamma_n}_{ij})-f'(x_i^ny_j^n)\right)=\alpha_i^n+\beta_j^n
\qquad \text{for\ }(i,j)\in \supp x^n\otimes y^n.
\end{equation}
If $(i,j)\in \supp \Pi$, then letting $n\to \infty$ leads to
\[
c_{ij}+\gamma \left(f'(\pi_{ij})-f'(x_iy_j)\right)=\lim_{n\to \infty}( \alpha_i^n+\beta_j^n).
\]
Moreover, if  $(i,j),(i,j')\in \supp \Pi$, then
\begin{align*}
\lim_{n \to \infty} \left(\beta^{n}_j -\beta^{n}_{j'}\right)
&=
c_{ij}+\gamma \left(f'(\pi_{ij})-f'(x_iy_j)\right)
-
\left\{ c_{ij'}+\gamma \left(f'(\pi_{ij'})-f'(x_iy_{j'})\right) \right\}, 
\end{align*}
hence the right-hand side is independent of $i$.
Similarly, if $(i,j),(i',j)\in \supp \Pi$, then
\[
c_{ij}+\gamma \left(f'(\pi_{ij})-f'(x_iy_j)\right)
-
\left\{ c_{i'j}+\gamma \left(f'(\pi_{i'j})-f'(x_{i'}y_{j})\right) \right\}
\]
is independent of $j$.
Thus  there exists $(\alpha, \beta)\in \mathbb{R}^{\supp \Pi}$ solving the linear equations of the form
\begin{equation}\label{lineareq}
\alpha_i+\beta_j =
c_{ij}+\gamma \left(f'(\pi_{ij})-f'(x_iy_j)\right) \quad (i,j)\in \supp \Pi.
\end{equation}

In the case  of $\gamma \in (0,\infty)$,
if  $(i,j),(i',j')\in \supp \Pi$, then $(i,j'),(i',j)\in \supp x\otimes y$  and
\begin{align*}
\lim_{n \to \infty}  
\left( f'( p^{\gamma_n}_{ij'})+f'( p^{\gamma_n}_{i'j}) \right)
=
\frac1{\gamma}\left( \alpha_i +\beta_j+ \alpha_{i'} +\beta_{j'} -c_{ij'}-c_{i'j}\right)
+ f'(x_iy_{j'})+f'(x_{i'}y_{j})
\end{align*}
by \eqref{lagn},  which implies $(i,j'), (i,j)\in \supp \Pi$, that is,  $\supp \Pi=\supp x\otimes y$.
Assume that  $\min\{|\supp x|,|\supp y|\}>1$, otherwise $|\Pi(x,y)|=1$ hence $\Pi=P^\gamma$ holds.
If we set   
\begin{align*}
\begin{split}
I&:=\{ i \ |\  i \in \supp x,\  i \neq  \max_{k\in \supp x}k\}, \quad
J:=\{ j \ |\  j \in \supp y,\  j \neq  \max_{k\in \supp y}k\}, \\
O&:=\left\{ Z  \in (0,1)^{I \times J}   \biggm|   \sum_{k\in J} z_{ik}<x_i, \quad   \sum_{k\in I} z_{kj}<y_j, \quad \text{for\ } (i,j)\in I\times J \right\},
\end{split}
\end{align*}
and define $\Phi:\overline{O}\to M_{N}(\mathbb{R})$ by 
\begin{equation}\label{coord}
\phi_{ij}(Z)
:=\begin{cases}
\displaystyle z_{ij} & \text{if\ } (i,j)\in I \times J,\\
\displaystyle
x_i-\sum_{k\in J} z_{ik} &\displaystyle  \text{if\ } j= \max_{k\in \supp y}k,  i\in I,\\
\displaystyle
y_j -\sum_{k\in I} z_{kj} & \displaystyle \text{if\ } i= \max_{k\in \supp x}k,  j\in J,\\
\displaystyle
1-\sum_{i \in I} x_i -\sum_{j \in J} y_j+ \sum_{(i,j)\in I\times J} z_{ij} &\displaystyle  \text{if\ } i= \max_{k\in \supp x}k,j= \max_{k\in \supp y}k, \\
0&\text{otherwise},
\end{cases}
\end{equation}
then $\Pi(x,y)=\Phi(\overline{O})$ and $P^\gamma,  \Pi \in \Phi(O)$.
Moreover,  $\mathcal{F}^\gamma_{x,y}\circ \Phi$ is strictly convex on~$O$ and  
\[
\nabla  \mathcal{F}^\gamma_{x,y} \circ \Phi |_{\Phi^{-1}(\Pi)}=0
\]
by \eqref{lineareq}.
Thus $\Pi$ is a unique minimizer of $\mathcal{F}^\gamma_{x,y}\circ \Phi $ on $O$, that is, $\Pi=P^\gamma$.

Assume $\gamma=0$.
Then $\supp \Pi$ is $C$-cyclically monotone  by \eqref{lineareq} hence $\Pi$ is optimal by Proposition~\ref{cmono}.
Furthermore,
multiplying \eqref{lagn} by $\pi_{ij}$ (resp.\,$\pi_{\ast}{}_{ij})$ and taking sum over $(i,j)\in \supp x\otimes y$ lead to
\begin{align*}
\sum_{(i,j)\in \supp \Pi }
\left(f'(p^{\gamma_n}_{ij})-f'(x_i^ny_j^n)\right)\pi_{ij}
&
=\frac{1}{\gamma_n}\left( \sum_{i\in \supp x } \alpha_i^n x_i +\sum_{ j\in \supp y } \beta_j^n y_j-\mathcal{C}(x,y)\right)\\
&=
\sum_{(i,j)\in \supp \Pi_\ast }
\left(f'(p^{\gamma_n}_{ij})-f'(x_i^ny_j^n)\right)\pi_\ast{}_{ij}.
\end{align*}
Since the left-hand side converges as $n\to \infty$, 
so does the right-hand side,
which implies $\supp \Pi_\ast \subset \supp \Pi$ and
\[
\sum_{(i,j)\in \supp \Pi }\left(f'(\pi_{ij})-f'(x_iy_j)\right)(\pi_\ast{}_{ij}-\pi_{ij})=0.
\]
For $t\in [0,1]$,  setting  $\Pi^t:=(1-t) \Pi +t \Pi_\ast \in \Pi(x,y)$, we find that
\[
\lim_{t \downarrow 0}\frac{d}{dt} \mathcal{D}_f(\Pi^t, x\otimes y)
=\sum_{(i,j)\in \supp \Pi }\left(f'(\pi_{ij})-f'(x_iy_j)\right)(\pi_\ast{}_{ij}-\pi_{ij})=0.
\]
By the strict convexity $\mathcal{D}_f(\cdot, x\otimes y)$ on $\Pi(x,y)$,  if $\Pi\neq \Pi_\ast$, then  $\mathcal{D}_f(\Pi^t, x\otimes y)$ is strictly increasing in $t\in[0,1]$
 hence
\[
\Lambda(x,y)\leq \mathcal{D}_f(\Pi, x\otimes y)=\mathcal{D}_f(\Pi^0, x\otimes y)
<\mathcal{D}_f(\Pi^1, x\otimes y)=\mathcal{D}_f(\Pi_\ast, x\otimes y)=\Lambda(x,y),
\]
which is a contradiction.
Thus  $\Pi=\Pi_\ast$ and the proof of Theorem~\ref{dconti} is complete.
\end{proof}
\noindent
We prove the monotonicity of the functions  on $(0,\infty)$ sending $\gamma$ to $\mathcal{C}^\gamma(x,y), \mathcal{D}_f(P^\gamma, x\otimes y)$. 
Note that the nondecreasing property of $\mathcal{C}^\gamma(x,y)$ is discussed in \cite{DPR}*{Property 6}.
\begin{theorem}\label{mono}
Given $x,y \in \mathcal{P}_N$ and $\gamma\in(0,\infty)$,
let $P^\gamma$ be an $(f,\gamma)$-coupling between $x,y$.
Then 
the two functions on $(0,\infty)$ sending $\gamma$ to $\mathcal{C}^\gamma(x,y), -\mathcal{D}_f(P^\gamma, x\otimes y)$
are increasing on $(0,\infty)$.
If  either $\mathcal{C}^{\gamma_0}(x,y)=\mathcal{C}^{\gamma_1}(x,y)$  or $\mathcal{D}_f(P^{\gamma_0}, x\otimes y)=\mathcal{D}_f(P^{\gamma_1}, x\otimes y)$ 
holds for $\gamma_0, \gamma_1\in (0,\infty)$, then $P^{\gamma_0}=P^{\gamma_1}$.

Moreover,  if $\supp x\otimes y$ is not $C$-cyclically monotone and $f'(0)=-\infty$, then  the both functions are  strictly increasing.
\end{theorem}
\begin{proof}
If there exist distinct $\gamma_0, \gamma_1\in  (0,\infty)$ such that 
\[
\mathcal{D}_f(P^{\gamma_0}, x\otimes y)=\mathcal{D}_f(P^{\gamma_1}, x\otimes y),
\]
then, for $\gamma:=(1-t)\gamma_0+t \gamma_1$ with $t\in (0,1)$, 
the strict convexity of $\mathcal{D}_f(\cdot, x\otimes y)$ leads to
\begin{align*}
\mathcal{F}^{\gamma}_{x,y}( P^{\gamma})
&\leq 
\mathcal{F}^{\gamma}_{x,y}( (1-t)P^{\gamma_0}+ tP^{\gamma_1})\\
&\leq  (1-t)\mathcal{F}^{\gamma}_{x,y}(P^{\gamma_0})+ t\mathcal{F}^{\gamma}_{x,y}(P^{\gamma_1}) 
= (1-t)\mathcal{F}^{\gamma_0}_{x,y}(P^{\gamma_0})+ t\mathcal{F}^{\gamma_1}_{x,y}(P^{\gamma_1})\\
&\leq 
 (1-t)\mathcal{F}^{\gamma_0}_{x,y}(P^{\gamma})+ t\mathcal{F}^{\gamma_1}_{x,y}(P^{\gamma})
=\mathcal{F}^{\gamma}_{x,y}(P^{\gamma}).
\end{align*}
Hence the inequalities above all become equalities 
and $P^\gamma= P^{\gamma_0}=P^{\gamma_1}$ by the uniqueness  of an $(f,\gamma)$-coupling.
Moreover, since $\mathcal{D}_f(P^\gamma, x\otimes y)$ is nonnegative, continuous in $\gamma>0$
and $\lim_{\gamma \uparrow \infty} \mathcal{D}_f(P^\gamma, x\otimes y)=0$,
we find that  $\mathcal{D}_f(P^{\gamma}, x\otimes y)$ is decreasing in $\gamma >0$.

For  $\gamma_0, \gamma_1 \in(0,\infty)$ with $\gamma_0<\gamma_1$,  the monotonicity of $\mathcal{D}_f(P^{\gamma}, x\otimes y)$ in $\gamma >0$ yields that 
\begin{align*}
\mathcal{C}^{\gamma_0}(x,y)+\gamma_0 \mathcal{D}_f(P^{\gamma_0}, x\otimes y)
=\mathcal{F}^{\gamma_0}_{x,y}(P^{\gamma_0}) 
\leq \mathcal{F}^{\gamma_0}_{x,y}(P^{\gamma_1})
\leq  \mathcal{C}^{\gamma_1}(x,y)+\gamma_1 \mathcal{D}_f(P^{\gamma_0}, x\otimes y).
\end{align*}
Thus the function on $(0,\infty)$ sending $\gamma$ to  $\mathcal{C}^\gamma(x,y)$  is  increasing.
If $\mathcal{C}^{\gamma_0}(x,y)=\mathcal{C}^{\gamma_1}(x,y)$,
then the inequalities above all become equalities, in particular $\mathcal{F}^{\gamma_0}_{x,y}(P^{\gamma_0})=\mathcal{F}^{\gamma_0}_{x,y}(P^{\gamma_1})$.
By the uniqueness of $(f,\gamma)$-couplings, $P^{\gamma_0}=P^{\gamma_1}$.
Suppose that $\supp x\otimes y$ is not $C$-cyclically monotone and $f'(0)=-\infty$.
In addition, we assume that 
 there exist $\gamma_0, \gamma_1\in  (0,\infty)$ with $\gamma_0<\gamma_1$ such that 
either $\mathcal{C}^{\gamma_0}(x,y)=\mathcal{C}^{\gamma_1}(x,y)$ 
or
$\mathcal{D}_f(P^{\gamma_0}, x\otimes y)=\mathcal{D}_f(P^{\gamma_1}, x\otimes y)$ holds.
Then $P^\gamma=P^{\gamma_0}$ for $\gamma\in[\gamma_0, \gamma_1]$.
For $\gamma\in [\gamma_0, \gamma_1]$,  it follows from Lemma~\ref{int} that  
$\psi_{ij}:=f'(p^\gamma_{ij})-f'(x_iy_j)$ is well-defined  for $(i,j)\in \supp x\otimes y$ and is independent of $\gamma$.
Furthermore,  there exists $(\alpha^\gamma, \beta^\gamma) \in \mathbb{R}^{\supp x\otimes y}$ such that 
\begin{equation}\label{gamma}
c_{ij}+\gamma \psi_{ij}=\alpha_i^\gamma+\beta_j^\gamma \quad \text{for\ }(i,j)\in \supp x\otimes y.
\end{equation}
Since $\supp x\otimes y$ is not $C$-cyclically monotone,  we find that  $\min\{|\supp x|, |\supp y|\}>1$.
Set $i_0:=\max_{i\in \supp x} i$ and $j_0:=\max_{j \in \supp y} j$.
Then it turns out that 
\begin{align*}
\alpha_i^\gamma+\beta_j^\gamma
=c_{ij_0}+c_{i_0 j}-c_{i_0j_0}
+\gamma \left(\psi_{ij_0} + \psi_{i_0j}-\psi_{i_0j_0} \right) 
\quad \text{for\ }(i,j)\in \supp x\otimes y.
\end{align*}
This with~\eqref{gamma} yields that
\[
c_{ij}+c_{i_0j_0}-\left(c_{ij_0}+c_{i_0 j}\right)
=-\gamma\left\{ \psi_{ij}+\psi_{i_0j_0}-\left(\psi_{ij_0}+\psi_{i_0 j}\right) \right\}
\qquad \text{for\ }(i,j)\in \supp x\otimes y.
\]
Since $\gamma\in [\gamma_0, \gamma_1]$ is arbitrary, 
$c_{ij}+c_{i_0j_0}=c_{ij_0}+c_{i_0 j}$ holds for  $(i,j)\in \supp x\otimes y$.
Moreover, it is possible to replace $(i_0, j_0)$ with any element $(i,j)\in \supp x\otimes y$,
consequently 
\[
c_{ij}+c_{kl}=c_{il}+c_{k j} \quad \text{for\ } (i,j), (k,l)\in \supp x\otimes y.
\]
This leads to the $C$-cyclical monotonicity of   $\supp x\otimes y$, which is a contradiction.
Thus the two functions  on $(0,\infty)$ sending $\gamma$ to $\mathcal{C}^\gamma(x,y), \mathcal{D}_f(P^\gamma, x\otimes y)$
are strictly  increasing.
\end{proof}

We prepare a lemma to provide an equality condition in  $0 \leq \mathcal{D}_f(P^\gamma,  x\otimes y )\leq \Lambda(x,y)$.
\begin{lemma}\label{full}
For $x,y \in \mathcal{P}_N$, the  following three conditions are equivalent to each other.
\begin{enumerate} 
\item
$x\otimes y$ is an $(f,\gamma)$-coupling  for some $\gamma>0$.
\item
$x\otimes y$ is an $(f,\gamma)$-coupling  for any $\gamma>0$.
\item
$\supp x \otimes y$ is $C$-cyclically monotone.
\end{enumerate}
\end{lemma}
\begin{proof}
The implication from (2) to (1) is trivial.
Assume (1).
By Lemma~\ref{int},  there exists $(\alpha, \beta)\in \mathbb{R}^{\supp x\otimes y}$ satisfying~\eqref{lag} 
for $P^\gamma=x\otimes y$, that is, $c_{ij}=\alpha_i+\beta_j$ for $(i,j)\in \supp x\otimes y$.
This leads to the $C$-cyclical monotonicity of $\supp x\otimes y$, hence (3) holds.

Assume (3).  By Proposition~\ref{cmono},  $x\otimes y$ is optimal.
For $\Pi \in \Pi(x,y)$ and $\gamma>0$,
we have
\begin{align*}
\mathcal{F}^\gamma_{x,y}(x\otimes y)
=\langle C, x\otimes y \rangle+ \gamma \mathcal{D}_f(x\otimes y,x\otimes y)
=\mathcal{C}(x,y) 
\leq \langle C, \Pi \rangle \leq  \mathcal{F}^\gamma_{x,y}(\Pi).
\end{align*}
Thus $x\otimes y$ is an $(f,\gamma)$-coupling for any $\gamma>0$, that is, (2) holds.
\end{proof}

\begin{proposition}\label{second}
For $x,y \in \mathcal{P}_N$ and $\gamma>0$,
the  following two conditions are equivalent to each other.
\begin{enumerate} 
\item
$\supp x\otimes y$ is $C$-cyclically monotone.
\item
 $\mathcal{D}_f(P^\gamma,  x\otimes y )=0$.
\end{enumerate}
Moreover, these conditions lead to the following condition.
\begin{enumerate} 
\item[(3)]
 $\mathcal{D}_f(P^\gamma,  x\otimes y )=\Lambda(x,y)$.
\end{enumerate}
If  $f'(0)=-\infty$,  then the condition $(3)$ is equivalent to the other conditions $(1)$ and~$(2)$.
\end{proposition}
\begin{proof}
If $\supp x\otimes y$ is $C$-cyclically monotone, then
\[
 \mathcal{D}_f(P^\gamma,  x\otimes y )=\Lambda(x,y)=0
 \quad \text{for\ }\gamma>0
\]
by Proposition~\ref{cmono} and Lemma~\ref{full}.
Thus the implication from (1) to (2) and (3) holds.

Conversely, for $\gamma>0$, if $\mathcal{D}_f(P^\gamma,  x\otimes y )=0$, 
then $P^\gamma=x\otimes y$ hence $\supp x\otimes y$ is $C$-cyclically monotone by  Lemma~\ref{full}.
Thus the implication from (2) to (1) holds.

Assume  $f'(0)=-\infty$ and fix $\gamma>0$.
If  $\supp x\otimes y$ is not $C$-cyclically monotone,  then $P^\gamma$ is not optimal by Proposition~\ref{cmono} and Lemma~\ref{int}
hence  $\mathcal{C}^\gamma(x,y) > \mathcal{C}(x,y)$.
Let $\Pi_\ast \in \Pi(x,y)$ an optimal coupling such that 
$\mathcal{D}_f(\Pi_\ast , x\otimes y)=\Lambda(x,y)$.
It turns out that
\begin{align*}
\mathcal{C}^\gamma(x,y) +\gamma \mathcal{D}_f(P^\gamma, x\otimes y)
&=\mathcal{F}^\gamma_{x,y}(P^\gamma)
<\mathcal{F}^\gamma_{x,y}(\Pi_\ast)
=\mathcal{C}(x,y)+\gamma \Lambda(x,y)
\end{align*}
by definition.
Thus  $\mathcal{D}_f(P^\gamma,  x\otimes y )<\Lambda(x,y)$, that is, (3) implies (1).
This concludes the proof of the proposition.
\end{proof}
\begin{remark}
For the strict monotonicity of $\mathcal{C}^\gamma(x,y) ,\mathcal{D}_f(P^\gamma,x\otimes y)$ with respect to $\gamma>0$
in Theorem \ref{mono}
and the implication from (3) to (1) in Proposition~\ref{second}, the assumption  $f'(0)=-\infty$ is necessary.
Indeed, in the setting as in Remark~\ref{38}, for $\gamma \in (0, 2/M_a)$ with $a\neq 0$,
$P^\gamma=P^{a,a}\in \Pi(x,y^a)$ is an unique optimal coupling, hence $\mathcal{D}_f(P^\gamma, x\otimes y)=\Lambda(x,y)$  holds.
However $\supp x\otimes y^a$ is not $C$-cyclically monotone.
\end{remark}
\begin{remark}
It is easy to  see that $\mathcal{C}$ is convex on $(\mathcal{P}_N \times \mathcal{P}_N, \|\cdot\|_2)$.
However, $\mathcal{C}^\gamma$ is not convex on $(\mathcal{P}_N \times \mathcal{P}_N, \|\cdot\|_2)$ in general.
For example,  we choose  $y^b \in \mathcal{P}_2$ and $Q^{b,s}\in \mathcal{P}_{2\times 2}$~as 
\begin{align*}
&y_1^b=b, \quad
y_2^b=1-b,\quad
\text{where}\ b\in [1/2,1],\\
&q_{11}^{b,s}=s,\quad q_{12}^{b,s}=q_{21}^{b,s}=b-s, \quad q_{22}^{b,s}=1-2b+s,
\quad
\text{where}\ s\in [2b-1,b],
\end{align*}
respectively.
Then  $\Pi(y^b, y^b)=\{Q^{b,s}\}_{s\in [2b-1,b]}$, especially  $\Pi(y^1,y^1)=\{Q^{1,1}\}$.
We take $C=(1-\delta_{ij})_{1\leq i,j \leq 2}\in M_2(\mathbb{R})$ and $f(r)=r\log r$.
For $\gamma>0$ and  $s\in  (2b-1,b)$ with $b\neq 1$,  we find that
\begin{align*}
\frac{\partial}{\partial s}
\mathcal{F}^{\gamma}_{x,y}(Q^{b,s})
=-2+\gamma  \log\frac{s(1-2b+s)}{(b-s)^2}.
\end{align*}
This yields that an unique $(f,\gamma)$-coupling between $y^b, y^b$ is $Q^{b, s(b)}$,
where 
\[
s(b):=\frac{1}{ 2(e^{2/\gamma}-1) }
\left(2e^{2/\gamma}b-2b+1-\sqrt{ 4(e^{2/\gamma}-1)b(1-b)+1 } \right).
\]
This means that $\mathcal{C}^\gamma(y^b,y^b)=2(b-s(b))$ for $b\in [1/2,1)$.
It is natural to set $s(1):=1$.

If $\mathcal{C}^\gamma$ is convex, then, for $b_t:=(1-t)b_0+ tb_1$ with $b_0, b_1 \in[1/2,1]$ and $t\in(0,1)$, 
\begin{align*}
2\left(b_t -s(b_t )\right)
=\mathcal{C}^\gamma(y^{b_t} ,y^{b_t})
&=
\mathcal{C}^\gamma((1-t)(y^{b_0},y^{b_0})+t(y^{b_1},y^{b_1}))\\
&\leq
(1-t)\mathcal{C}^\gamma(y^{b_0},y^{b_0})
+
t\mathcal{C}^\gamma(y^{b_1},y^{b_1})\\
&=2(1-t)\left( b_0 -s(b_0) \right)
+2t\left( b_1 -s(b_1) \right)
\end{align*}
holds, which is equivalent to the concavity of $s(b)$ on $b\in [1/2,1]$.
However, we find that 
$\lim_{b\uparrow 1}s''(b)=2e^{2/\gamma}>0$.
Thus $\mathcal{C}^\gamma$ is not convex.
\end{remark} 

\section{Gradient descent on a space of couplings}
We construct an iterative process to find the relaxed minimizer via the Bregman divergences associated to 
a strictly convex, continuous function $f$ on $[0,1]$  such that $f\in C^2((0,1])$ with $f'(0)=-\infty$,
where the output $\Pi\in \mathcal{P}_{N\times N}$ is always a coupling between two given points in $\mathcal{P}_N$ even if we stop the iteration at the finite step.
Note that $f(r)=r \log r $ satisfies this condition.

For this purpose,  we regard $\Pi(x,y)$ as a submanifold of $M_{N}(\mathbb{R})$ and consider the gradient of 
$\mathcal{F}^\gamma_{x,y}$ with respect to the induced Riemannian metric.
Note that an algorithm for  general convex functions is already introduced (for instance, see~
\cite{DPR}*{Section~4} and \cite{PC}*{Remark~4.10, Section~4.6} and the references therein)
and  a gradient descent of a convex function on $\mathcal{P}_N$ is mentioned  in \cite{PC}*{Section~9.3}.
However these are different from our method.

%
In what follows,  we fix  $x,y\in \mathcal{P}_N$.
For simplicity, we assume $|\supp x|=|\supp y|=N$.
Let $\Phi:O\to M_N(\mathbb{R})$ be as in \eqref{coord}.
Then $M:=\Phi(O)\subset \Pi(x,y)$ becomes a submanifold of $M_N(\mathbb{R})$ of dimension $(N-1)^2$, 
where $(M, \Phi^{-1}|_M)$ determines a global coordinate system of $M$.
Let $(\partial_{ij})_{1\leq i,j \leq N}$ be coordinate vector fields of this  global coordinate system.
We denote by $g$  the induced Riemannian metric on $M$.
We see that 
\[
g_{(i,j)(k,l)}:=g\left({\partial_{ij}},{\partial_{kl}}\right)=(1+\delta_{ik})(1+\delta_{jl})
\quad \text{for}\ 1\leq i,j,k,l\leq N-1,
\]
and its inverse matrix, denoted by $(g^{(i,j)(k,l)})_{1\leq i,j,k,l\leq N-1}$,  is 
\[
g^{(i,j)(k,l)}:=\left(\frac1N-\delta_{ik}\right)\left(\frac1N-\delta_{jl}\right)
\quad \text{for}\ 1\leq i,j,k,l\leq N-1.
\]
This ensures that  $M$ is a totally geodesic submanifold of $M_N(\mathbb{R})$,
that is, $(\Pi^t)_{t\in [0,1]}$ is a geodesic in $M$ if and only if  there exist $Z_0, Z_1\in O$ such that 
\[
\Pi^t=\Phi \left( (1-t) Z_0+ t Z_1 \right)\quad \text{for}\ t\in [0,1]. 
\]
Fix $\gamma>0$  and  a strictly convex, continuous function $f$ on $[0,1]$  such that $f\in C^2((0,1])$.
Then  $\mathcal{F}^{\gamma}_{x,y}\in C^2(M)$
and the composition of $\mathcal{F}^{\gamma}_{x,y}$ and a nonconstant geodesic on  $(M,g)$ is  always strictly convex.
Hence the function $\mathcal{F}:O\to \mathbb{R}$ defined by 
\[
\mathcal{F}(Z):=\mathcal{F}^\gamma_{x,y}(\Phi(Z))
\]
is $C^2$, strictly convex on $O$.
Moreover, $\mathcal{F}$ is continuously extended to the closure $\overline{O}$ of $O$.
For  $1\leq i,j \leq N-1$,  we write 
\begin{align*}
\nabla_{ij}\mathcal{F}:&=\left(\frac{\partial \mathcal{F}}{\partial z_{ij}}\right)
=\sum_{(k,l)=(i,j),(N,N),(i,N),(N,j)}\mathrm{sgn}(k,l)\left\{ c_{kl}+\gamma\big( f'( \phi_{kl})-f'(x_ky_l) \big)\right\}
\end{align*}
and $\nabla \mathcal{F}:=(\nabla_{ij}\mathcal{F})_{1\leq i,j \leq N-1}$, 
where $\mathrm{sgn}(k,l):=1$ if $(k,l)=(i,j),(N,N)$, otherwise $\mathrm{sgn}(k,l):=-1$.
Then $\nabla \mathcal{F}(Z)=0$ if and only if $\Phi(Z)$ is an $(f,\gamma)$-coupling between~$x,y$.
To define a sequence $(Z^n)_{n\in \mathbb{N}}$ in $O$,  we prepare a lemma.
\begin{lemma}\label{epsilon}
Let $f$ be a strictly convex, continuous function on $[0,1]$ so that $f\in C^2((0,1])$ with $f'(0)=-\infty$.
Given $\gamma>0$ and  $x,y\in \mathcal{P}_N$ with  $|\supp x|=|\supp y|=N$, 
define  functions $\varepsilon_{ij},  \varepsilon^i, \varepsilon_j, \varepsilon: O\to \mathbb{R}$  for $1\leq i,j \leq N-1$ by
\begin{align*}
\varepsilon_{ij}(Z)
:&=\begin{cases}
\dfrac{z_{ij}}{2\nabla_{ij}\mathcal{F}(Z) }& \text{if}\ \nabla_{ij}\mathcal{F}(Z)>0,\\
1&\text{otherwise},
\end{cases}\\[-2pt]
\varepsilon^{i}(Z)
:&=\begin{cases}
 \dfrac{x_i-\sum_{k=1}^{N-1}z_{ik}}{-2 \sum_{k=1}^{N-1} \nabla_{ik}\mathcal{F}(Z)}
 &\displaystyle  \text{if}\ \sum_{k=1}^{N-1}
 \nabla_{ik}\mathcal{F}(Z)<0,\\
1&\text{otherwise},
\end{cases}\\[-2pt]
\varepsilon_j(Z)
:&=\begin{cases}
\displaystyle \dfrac{y_j-\sum_{k=1}^{N-1}z_{kj}}{- 2\sum_{k=1}^{N-1} \nabla_{kj}\mathcal{F}(Z)}
&\displaystyle  \text{if}\ \sum_{k=1}^{N-1} \nabla_{kj}\mathcal{F}(Z)<0,\\
1&\text{otherwise},
\end{cases}\\[-2pt]
%
\varepsilon(Z):&=\min_{1\leq i,j\leq N}\{\varepsilon^i(Z), \varepsilon_j(Z),\varepsilon_{ij}(Z) \},
\end{align*}
respectively.
Then  $Z-t \nabla\mathcal{F}(Z) \in O$ for $t\in[0,\varepsilon(Z)]$
and
\[
H(Z)
:=\max_{t\in [0, \varepsilon(Z)]} \frac{d^2}{dt^2}\mathcal{F}(Z-t\nabla \mathcal{F}(Z)) \geq0
\]
with equality if and only if $\nabla \mathcal{F}(Z)=0$.
\end{lemma}
\begin{proof}
Set  $Z^t:=Z-t \nabla \mathcal{F}(Z)$.
For $1\leq i, j\leq N-1$ and  $t\in[0,\varepsilon(Z)]$, 
we see that 
\begin{align*}
z^t_{ij}
&\geq \begin{cases}
\dfrac12z_{ij}& \text{if}\ \nabla_{ij}\mathcal{F}(Z)>0,\\
z_{ij} &\text{otherwise},
\end{cases}\\[-2pt]
\sum_{k=1}^{N-1} z^t_{ik}&
=\sum_{k=1}^{N-1} z_{ik}-t \sum_{k=1}^{N-1} \nabla_{ik}\mathcal{F}(Z)
\leq \begin{cases}
\displaystyle\dfrac12\left(x_i + \sum_{k=1}^{N-1} z_{ik} \right) & \text{if}\ \displaystyle  \sum_{k=1}^{N-1} \nabla_{ik}\mathcal{F}(Z)<0, \\
\displaystyle \sum_{k=1}^{N-1} z_{ik} &\text{otherwise},
\end{cases}\\[-2pt]
\sum_{k=1}^{N-1} z^t_{kj}&
=\sum_{k=1}^{N-1} z_{kj}-t \sum_{k=1}^{N-1} \nabla_{kj}\mathcal{F}(Z)
\leq \begin{cases}
\displaystyle\dfrac12\left(y_j + \sum_{k=1}^{N-1} z_{kj} \right) & \text{if}\ \displaystyle  \sum_{k=1}^{N-1} \nabla_{kj}\mathcal{F}(Z)<0, \\
\displaystyle \sum_{k=1}^{N-1} z_{kj} &\text{otherwise},
\end{cases}
\end{align*}
implying $Z^t \in O$.
Since $\mathcal{F}$ is strictly convex on $O$, 
$H(Z) \geq 0$ with equality if and only if  $\nabla \mathcal{F}(Z)=0$.
\end{proof}
We compute that
\begin{align*}
\frac{d^2}{dt^2}\mathcal{F}(Z-t\nabla \mathcal{F}(Z))
&= \frac{\partial^2\mathcal{F}}{\partial z_{ij}\partial z_{kl}}(Z-t\nabla \mathcal{F}(Z))
    \cdot \nabla_{ij}\mathcal{F}(Z) \cdot  \nabla_{kl}\mathcal{F}(Z),\\
\frac{\partial^2\mathcal{F}}{\partial z^{ij}\partial z^{kl}}
&=  
\gamma\big( f''(\phi_{ij} )\delta_{ik}\delta_{jl}   +f'(\phi_{NN} )+f''(\phi_{iN} ) \delta_{ik}+ f''(\phi_{Nj} )\delta_{jl}\big).
\end{align*}
\begin{theorem}\label{descent}
With the same assumptions and notation as in Lemma~\ref{epsilon}, 
define a sequence $(Z^n)_{n\in \mathbb{N}}\subset O$ inductively as follows.:
Let $Z^1:=\Phi^{-1}(x\otimes y)=( x_i y_j)_{1\leq i,j \leq N-1}$.
If $Z^n\in O$ has been defined, let  $Z^{n+1}:=Z^n$ if $\nabla \mathcal{F}(Z^n)=0$, 
and otherwise let 
\[
Z^{n+1}:=Z^n-t_n \nabla \mathcal{F}(Z^n), \quad \text{where}\ t_n:=\min\left\{\frac{\| \nabla \mathcal{F}(Z^n)\|_2^2}{H(Z^n)}, \varepsilon(Z^n)\right\}.
\]
Then  $Z^\infty:=\lim_{n\to \infty} Z^n$ exists and $\Phi(Z^\infty)$  is  an $(f,\gamma)$-optimal coupling between $x,y$.
\end{theorem}
\begin{proof}
It is enough to show the case that $\nabla \mathcal{F}(Z^n) \neq 0$ for any $n\in \mathbb{N}$.
The Taylor expansion of $\mathcal{F}$ implies 
\begin{align*}
\mathcal{F}(Z^{n+1})-\mathcal{F}(Z^n) 
&\leq
-t_n  \langle \nabla \mathcal{F}(Z^n), \nabla \mathcal{F}(Z^n)  \rangle 
+
\frac{t_n^2}{2} \max_{t\in [0, t_n]} \frac{d^2}{dt^2}\mathcal{F}(Z^n-t \nabla \mathcal{F}(Z^n))\\
&\leq -t_n \|\nabla \mathcal{F}(Z^n)\|_2^2 +\frac{t_n^2}{2} H(Z^n)
\leq -\frac{t_n}2\|\nabla \mathcal{F}(Z^n)\|_2^2<0.
\end{align*}
Thus $(\mathcal{F} (Z^n))_{n\in \mathbb{N}}$ is a strictly decreasing sequence.
By
\[
\inf_{n\in \mathbb{N}} \mathcal{F} (Z^n)
\geq \mathcal{F}(\Phi^{-1}(P^\gamma)),
\]
where $P^\gamma$  is  an $(f,\gamma)$-optimal coupling between $x,y$,
the limit $\mathcal{F}^\infty:=\lim_{n\to \infty} \mathcal{F}(Z^n)$ exists. 
Then
\begin{equation}\label{cauchy}
0=\lim_{n,L\to \infty}\left|\mathcal{F}(Z^{n+L})-\mathcal{F}(Z^n)\right|
\geq 
\lim_{n,L\to \infty} \sum_{l=0}^{L-1}\frac{t_{n+l}}2\|\nabla \mathcal{F}(Z^{n+l})\|_2^2.
\end{equation}
%
%
If $\liminf_{n\to \infty}\|\nabla \mathcal{F}(Z^{n})\|_2=0$, 
then there exists a subsequence of $(Z^{n})$ converging to $P^\gamma$.
This implies that $\mathcal{F}^\infty=0$, hence $\lim_{n\to \infty}Z^{n}=P^\gamma$.

Assume $\liminf_{n\to \infty}\|\nabla \mathcal{F}(Z^{n})\|_2\in (0,\infty]$. 
Then  it follows from \eqref{cauchy} that 
\begin{align*}
\lim_{n,L\to \infty}\sum_{l=0}^{L-1}t_{n+l}&=0,\\
\|Z^{n+L}-Z^{n}\|_2
&=
\left\| \sum_{l=0}^{L-1} t_{n+l} \nabla \mathcal{F}(Z^{n+l})  \right\|_2
\leq  \sum_{l=1}^{L} t_{n+l}\left\| \nabla \mathcal{F}(Z^{n+l})\right\|_2\\
&\leq  \left(\sum_{l=0}^{L-1} t_{n+l}\right)^{\frac12}\left(  \sum_{l=0}^{L-1} t_{n+l}\left\| \nabla \mathcal{F}(Z^{n+l})\right\|_2^2\right)^{\frac12}
\xrightarrow{n,L\to \infty}0,
\end{align*}
hence $Z^{\infty}:=\lim_{n\to \infty}Z^{n}\in \overline{O}$ exists.
By the assumption that $\liminf_{n\to \infty}\|\nabla \mathcal{F}(Z^{n})\|_2^2\neq 0$,
$\Phi(Z^\infty)\neq P^\gamma$ holds.
Since all of $\varepsilon, H, \nabla \mathcal{F}:O \to \mathbb{R}$ are continuous, 
if $Z^\infty\in O$,  then
\[
\inf_{n\in \mathbb{N}} t_n:=\inf_{n\in \mathbb{N}} \min\left\{\frac{\|\nabla \mathcal{F}(Z^n)\|_2^2}{H(Z^n)}, \varepsilon(Z^n)\right\}
>0,
\]
which is a contradiction to  $\lim_{n\to \infty} t_n=0$.
Thus $Z^\infty\in \partial O$ holds, that is, there exists $1\leq i,j\leq N$ such that $\phi_{ij}(Z^\infty)=0$.
If $1\leq i,j\leq N-1$ and $\phi_{iN}(Z^\infty),\phi_{Nj}(Z^\infty)>0$,
then 
\[
\lim_{n\to \infty} \nabla_{ij}\mathcal{F}(Z^n)=-\infty,
\]
consequently $(z_{ij}^n=\phi_{ij}(Z^n))_{n\in \mathbb{N}}$ is an increasing sequence,
which is a contradiction to $\phi_{ij}(Z^\infty)=0$.
The other cases are similar.
Thus $\lim_{n\to \infty}\|\nabla \mathcal{F}(Z^{n})\|_2^2 =0$ and $Z^\infty =P^\gamma$ follow.
This completes the proof of the theorem.
\end{proof}
%
In Theorem~\ref{descent}, we consider a gradient descent of $\mathcal{F}$ in $O$.
One can give a similar discussion for a gradient descent of $\mathcal{F}^\gamma_{x,y}$ in $M$.
To do this,  we identify the tangent space $T_{\Pi} M$ at $\Pi \in M$ with  $M_{N-1}(\mathbb{R})$ by a natural isomorphism 
\[
\sum_{i,j=1}^{N-1}  \zeta^{ij}{\partial_{ij}}\bigg|_{\Pi} \in T_{\Pi} M
\longleftrightarrow
(\zeta^{ij})_{1\leq i,j\leq N-1} \in M_{N-1}(\mathbb{R}).
\]
Then  the gradient of $\mathcal{F}^{\gamma}_{x,y}$, denoted by  $\nabla_g \mathcal{F}^{\gamma}_{x,y}$,  
at $\Phi(Z)\in M$ is identified with 
\begin{align*}
D(Z)=(d^{ij}(Z))_{1\leq i,j\leq N-1}
:&=\left(\sum_{k,l=1}^{N-1}g^{(i,j) (k,l) } \nabla_{ij}\mathcal{F} (Z) \right)_{1\leq i,j\leq N-1},
\end{align*}
where $D(Z)=0$ if and only if $\Phi(Z)$ is an $(f,\gamma)$-coupling between~$x,y$.
Note that 
\begin{align*}
G(Z):=g(\nabla_g \mathcal{F}^{\gamma}_{x,y},\nabla_g \mathcal{F}^{\gamma}_{x,y}) (\Phi(Z))
&=\sum_{i,j,k,l=1}^{N-1} 
g_{(i,j)(k,l)} d^{ij}(Z) d^{kl}(Z) \\
&=\sum_{i,j,k,l=1}^{N-1} 
g^{(i,j)(k,l)}\nabla_{ij} \mathcal{F} (Z) \nabla_{kl} \mathcal{F} (Z).
\end{align*}
For a geodesic $\Pi^t$ in $M$ defined by
\[
\Pi^t:=\exp_{\Pi}\left(-t\nabla \mathcal{F}^{\gamma}_{x,y}(\Pi)\right)  =\Phi\left(Z -t D(Z)\right ),
\quad\text{where\ }Z:=\Phi^{-1}(\Pi),
\]
we observe that
\begin{align*}
 \frac{d^2}{dt^2}\mathcal{F}^{\gamma}_{x,y}(\Pi^t)
=\sum_{i,j,k,l=1}^{N-1}
\frac{\partial^2 \mathcal{F}}{\partial z_{ij} \partial z_{kl}}(Z -t D(Z))  
\cdot d^{ij}(Z)
\cdot d^{kl}(Z).
\end{align*}
The following corollary is proved in analogy with Theorem~\ref{descent}.
\begin{corollary}\label{cordescent}
With the same assumptions  as in Lemma~\ref{epsilon}, 
for $1\leq i,j \leq N-1$, 
define  functions $\epsilon_{ij},  \epsilon^i, \epsilon_j, \epsilon: O\to \mathbb{R}$ by
\begin{align*}
\epsilon_{ij}(Z)
:&=\begin{cases}
\dfrac{z_{ij}}{2d^{ij}(Z)}& \text{if}\ d^{ij}(Z)>0,\\
1&\text{otherwise},
\end{cases}\\
\epsilon^{i}(Z)
:&=\begin{cases}
 \dfrac{x_i-\sum_{k=1}^{N-1}z_{ik}}{-2 \sum_{k=1}^{N-1} d^{ik}(Z)}&\displaystyle  \text{if}\ \sum_{k=1}^{N-1} d^{ik}(Z)<0,\\
1&\text{otherwise},
\end{cases}\\
\epsilon_j(Z)
:&=\begin{cases}
\displaystyle \dfrac{y_j-\sum_{k=1}^{N-1}z_{kj}}{- 2\sum_{k=1}^{N-1} d^{kj}(Z)}&\displaystyle  \text{if}\ \sum_{k=1}^{N-1} d^{kj}(Z)<0,\\
1&\text{otherwise},
\end{cases}\\
\epsilon(Z):&=\min_{1\leq i,j\leq N}\{\epsilon^i(Z), \epsilon_j(Z),\epsilon_{ij}(Z) \},
\end{align*}
respectively.
Then  $Z-t D(Z)\in O$ for $t\in[0,\epsilon(Z)]$
and
\[
D_2(Z)
:=\max_{t\in [0, \epsilon(Z)]} 
 \frac{d^2}{dt^2}\mathcal{F}^{\gamma}_{x,y}(\Phi(Z-t D(Z))   ) \geq 0
\]
with equality if and only if $D(Z)=0$.

Furthermore,  define $(Z^n)_{n\in \mathbb{N}}\subset O$ inductively as follows.:
Let $Z^1:=\Phi^{-1}(x\otimes y)$.
If $Z^n\in O$ has been defined, let  $Z^{n+1}:=Z^n$ if $D(Z^n)=0$, 
and otherwise let 
\[
Z^{n+1}:=Z^n-t_n D(Z^n), \quad \text{where}\ t_n:=\min\left\{\frac{G(Z^n)}{D_2(Z^n)}, \epsilon(Z^n)\right\}.
\]
Then  $Z^\infty:=\lim_{n\to \infty} Z^n$ exists and $\Phi(Z^\infty)$  is  an $(f,\gamma)$-optimal coupling between $x,y$.
\end{corollary}

\qquad

\begin{ack}
The author would like to express her gratitude to Jun Kitagawa for his suggestions and discussions.
She also thanks Han Bao and Sho Sonoda for their useful comments.

The author was supported in part by KAKENHI 19K03494, 19H01786.
\end{ack}

\qquad

\end{document}